\documentclass[11pt]{article}

\usepackage[T1]{fontenc}
\usepackage[utf8]{inputenc}
\usepackage{lmodern}
\usepackage{microtype}
\usepackage[a4paper,margin=30mm]{geometry}
\usepackage{amsmath,amssymb,amsthm,mathtools}
\usepackage{enumitem}
\usepackage{xcolor}
\usepackage{hyperref}
\usepackage{aliascnt}
\usepackage[nameinlink,capitalise,noabbrev]{cleveref}
\usepackage{fancyhdr}

\hypersetup{
  colorlinks=true,
  linkcolor=blue!45!black,
  citecolor=green!35!black,
  urlcolor=blue!55!black,
  pdftitle={Deformations of Kähler and Balanced Hyperbolicity},
  pdfsubject={Deformation stability, hyperbolic cohomology of principal torus bundles, and Lefschetz saturation},
  pdfkeywords={balanced hyperbolicity, Kähler hyperbolicity, principal torus bundles, deformations, Aeppli cohomology, hyperbolic cohomology}
}

\setlist{nosep,leftmargin=2em}
\allowdisplaybreaks

\newtheorem{theorem}{Theorem}[section]

\newaliascnt{proposition}{theorem}
\newtheorem{proposition}[proposition]{Proposition}
\aliascntresetthe{proposition}

\newaliascnt{lemma}{theorem}
\newtheorem{lemma}[lemma]{Lemma}
\aliascntresetthe{lemma}

\newaliascnt{corollary}{theorem}
\newtheorem{corollary}[corollary]{Corollary}
\aliascntresetthe{corollary}

\newaliascnt{criterion}{theorem}
\newtheorem{criterion}[criterion]{Criterion}
\aliascntresetthe{criterion}

\newaliascnt{example}{theorem}
\newtheorem{example}[example]{Example}
\aliascntresetthe{example}

\theoremstyle{definition}

\newaliascnt{definition}{theorem}
\newtheorem{definition}[definition]{Definition}
\aliascntresetthe{definition}

\newaliascnt{question}{theorem}
\newtheorem{question}[question]{Question}
\aliascntresetthe{question}

\theoremstyle{remark}

\newaliascnt{remark}{theorem}
\newtheorem{remark}[remark]{Remark}
\aliascntresetthe{remark}

\newaliascnt{warning}{theorem}

\aliascntresetthe{warning}

\newcommand{\ddbar}{\partial\bar\partial}
\newcommand{\dbar}{\bar\partial}

\newcommand{\DR}{\mathrm{DR}}
\newcommand{\BC}{\mathrm{BC}}
\newcommand{\Aep}{\mathrm{A}}
\newcommand{\im}{\operatorname{Im}}
\newcommand{\kerop}{\operatorname{Ker}}
\newcommand{\cB}{\mathcal B}
\newcommand{\Vhyp}{V_{\widetilde d}}

\newcommand{\R}{\mathbb R}
\newcommand{\C}{\mathbb C}
\newcommand{\Q}{\mathbb Q}
\newcommand{\Z}{\mathbb Z}
\newcommand{\D}{\mathbb D}

\title{\bfseries Deformations of K\"{a}hler and Balanced Hyperbolicity}
\author{Jixiang Fu\and Jingcao Wu}
\date{}

\begin{document}
\maketitle

\begin{abstract}
We study the deformation stability of K\"ahler and balanced hyperbolicity. We prove that K\"ahler hyperbolicity is open under arbitrary small holomorphic deformations, but need not be closed, even in families of smooth projective surfaces. By contrast, balanced hyperbolicity is not open in general: in every complex dimension at least five, we construct a balanced hyperbolic central fibre with non-balanced nearby fibres. We establish criteria for the deformation stability of balanced hyperbolicity using Aeppli cohomology and Lefschetz theory. For principal holomorphic torus bundles over products of curves of genus at least two, we compute the relevant hyperbolic cohomology, obtain a numerical characterization of balanced hyperbolicity, and prove its openness under small deformations.
\end{abstract}



\tableofcontents

\section{Introduction}

Let $(X,g)$ be a compact Riemannian manifold with universal cover $p:\widetilde X\to X$. A differential form $\gamma$ on $X$ is called $\widetilde d$-bounded if
\[
p^{*}\gamma=d\Gamma,\qquad \|\Gamma\|_{L^{\infty}(\widetilde X,p^{*}g)}<\infty
\]
for some form $\Gamma$ of one lower degree. Throughout, complex manifolds are assumed to be connected and of dimension $n\geq2$; we set $A^{p,q}=0$ when $p<0$, $q<0$, $p>n$, or $q>n$. Following \cite{Gro91,MP23}, a K\"{a}hler metric $\omega$ is \emph{K\"{a}hler hyperbolic} if $\omega$ is $\widetilde d$-bounded; a balanced metric is \emph{balanced hyperbolic} if $\omega^{n-1}$ is $\widetilde d$-bounded. We examine how these properties behave under small holomorphic deformations.

Our first result establishes openness of K\"ahler hyperbolicity without additional deformation hypotheses.

\begin{theorem}\label{thm:kahler-openness}
Let $\pi:\mathcal X\to\Delta$ be a proper holomorphic submersion with connected compact fibres. If $X_0$ is K\"ahler hyperbolic, then $X_t$ is K\"ahler hyperbolic for every sufficiently small $t$.
\end{theorem}

The proof in \cref{sec:monodromy} combines Deligne's semisimplicity and fixed-part theorems \cite{Del71} with Grothendieck's analytic Hilbert theorem \cite{GCartIX}. The argument first establishes openness in algebraic families and then passes to arbitrary analytic deformations. Closedness, however, fails even for projective surfaces.

\begin{theorem}\label{thm:kahler-not-closed}
There is a proper holomorphic submersion $q:\mathcal Z\to\Delta$ with connected smooth projective surface fibres such that $Z_s$ is K\"ahler hyperbolic for every $s\ne0$, while $Z_0$ contains a rational curve and is not K\"ahler hyperbolic.
\end{theorem}

The picture is different for balanced hyperbolicity. In every dimension $N\geq5$, the Iwasawa deformation, stabilized by products of curves of genus at least two, yields a balanced-hyperbolic central fibre with non-balanced nearby fibres (\cref{sec:counterexample}). This leads to the following question.

\begin{question}\label{q:balanced-locus}
Suppose $X_0$ is balanced hyperbolic and every sufficiently small $X_t$ is known to be
balanced. Must every sufficiently small $X_t$ be balanced hyperbolic?
\end{question}

Under an Ehresmann trivialization, de Rham cohomology and its subspaces of $\widetilde d$-bounded classes are fixed by the underlying smooth manifold. By contrast, Hodge type and positivity vary with the complex structure. This distinction guides our approach to the balanced deformation problem.

\paragraph{The balanced deformation problem.}
For $H=H^{2n-2}_{\DR}(X,\R)$ put $V=\Vhyp^{2n-2}(X)$, and let $\mathcal P_t\subset H$ consist of classes admitting a real closed $(n-1,n-1)$-representative on $X_t$. In \cref{sec:compact-program}, an Aeppli obstruction describes the intersection $V\cap\mathcal P_t$ and leads to a constant-rank continuation theorem that preserves positivity (\cref{thm:constant-rank}). In particular, the vanishing of $H^{n-2,n}_{\Aep}(X_0)$, equivalently $h^{2,0}_{\BC}(X_0)=0$, implies openness (\cref{thm:aeppli-vanishing}). We also answer \cref{q:balanced-locus} affirmatively for principal holomorphic torus bundles over products of curves of genus at least two (\cref{thm:torus-deform}), using the numerical criterion in \cref{thm:torus-numerical}.

\paragraph{Lefschetz saturation.} The known graded-ideal property of hyperbolic cohomology \cite[Remark~4.5]{BKS24}, combined with Lefschetz surjectivity, gives the saturation criterion in \cref{prop:lefschetz-saturation}. One consequence is:

\begin{theorem}\label{t14}
Let $\pi:\mathcal X\to\Delta$ be a holomorphic family of compact complex $n$-folds with $n\geq3$, and let $k$ be an integer with $1\leq k\leq n-2$. If $X_0$ is K\"ahler $k$-hyperbolic, then every balanced complex structure on the underlying smooth manifold is balanced hyperbolic. In particular, every sufficiently small holomorphic deformation $X_t$ of $X_0$ is balanced hyperbolic.
\end{theorem}

As in \cite[Definition~2.3]{FWW26}, K\"ahler $k$-hyperbolicity means that $X_0$ carries a K\"ahler form $\omega_0$ with $\omega_0^k$ $\widetilde d$-bounded. The same argument also controls higher K\"ahler powers.

\begin{theorem}\label{thm:kahler-k-upper}
Let $\pi:\mathcal X\to\Delta$ be a holomorphic family of compact complex $n$-folds. Suppose that $X_0$ is K\"ahler $r$-hyperbolic for some $1\leq r\leq n$, and let $\ell$ satisfy
\[
 r\leq \ell\leq n,\qquad 2\ell\geq n+r.
\]
Then every K\"ahler complex structure on the underlying smooth manifold is K\"ahler $\ell$-hyperbolic. In particular, every sufficiently small holomorphic deformation $X_t$ of $X_0$ is K\"ahler $\ell$-hyperbolic.
\end{theorem}

For $\ell=n-1$ and $r\leq n-2$, this gives $\Vhyp^{2n-2}(X)=H^{2n-2}_{\DR}(X,\R)$ and hence \cref{t14} via \cref{prop:intersection}. At the borderline $r=n-1$, the relevant multiplication map need not be surjective; our argument therefore does not establish the analogous saturation or balanced hyperbolicity openness. See \cref{sec:topological-program}.

These saturation statements also have consequences over the entire parameter disc. Let $\pi:\mathcal X\to\Delta$ be a proper holomorphic submersion with connected compact fibres of dimension $n\geq3$. If one fibre is K\"ahler $k$-hyperbolic, $1\leq k\leq n-2$, and uncountably many fibres are Moishezon, then every fibre is balanced hyperbolic; see part~\textup{(a)} of \cref{thm:whole-disc}. Here Rao's Moishezon propagation theorem \cite[Theorem~1.5]{Rao26} supplies balancedness throughout the family. In the ordinary K\"ahler hyperbolic case $k=1$, the additional Moishezon hypothesis is unnecessary: a single K\"ahler hyperbolic fibre already forces every fibre over the original disc to be Moishezon and balanced hyperbolic; see part~\textup{(b)} of \cref{thm:whole-disc}. The latter conclusion has a proof using the ampleness of the canonical bundle \cite[Theorem~2.11]{CY18} and Rao--Tsai's extension of bigness \cite[Corollary~4.3]{RT21}, independently of \cite{Rao26}. Neither statement asserts ordinary K\"ahler hyperbolicity of every fibre.

\paragraph{A supplementary bounded-correction problem.}
Appendix~\ref{sec:universal-program} examines when a bounded $(\partial+\dbar)$-potential can be corrected to a bounded $d$-primitive. This question is independent of the deformation theorems.

\noindent\emph{Relation to prior work.}
The K\"ahler openness argument uses Deligne's monodromy semisimplicity and the Hodge structure on global flat endomorphisms \cite[Th\'eor\`eme~4.2.6 and Corollaire~4.2.8(i)]{Del71}, together with Grothendieck's analytic relative-embedding and Hilbert-scheme constructions \cite{GCartVIII,GCartIX}. For related work on holomorphic deformations and Hodge classes see \cite{Khe25KH,LS26}. The nodal simultaneous-resolution construction follows the local method of \cite{Ati58}. The principal-torus-bundle deformation statement uses Catanese's theorem \cite[Proposition~4.2]{Cat04} and the deformation theory of products of curves \cite{vO05}; the cohomological calculation is given here. The whole-disc balanced hyperbolicity statements rely on \cite{AB93,CY18,Rao26,RT21}, as specified below. These results concern separate aspects of deformation stability and are not consequences of unrestricted K\"ahler openness.

\paragraph{Organization.} \cref{sec:counterexample} gives the counterexamples, and \cref{sec:monodromy} proves K\"ahler openness. \cref{sec:compact-program,sec:topological-program,sec:torus-question} address balanced hyperbolicity through Aeppli cohomology, Lefschetz saturation, and principal torus bundles. The independent analytic question appears in the appendix. In general, \cref{q:balanced-locus} remains open.

\section{Cohomological framework}
\label{sec:common}

\subsection{\texorpdfstring{$\widetilde d$-bounded}{d-tilde-bounded} cohomology}

Let $(X,g)$ be a compact Riemannian manifold and let $p:\widetilde X\to X$ be its universal cover, equipped with the lifted metric $\widetilde g=p^{*}g$.

\begin{definition}\label{def:dtilde}
Let $k\geq1$. A closed $k$-form $\alpha$ on $X$ is called \emph{$\widetilde d$-bounded} if there exists a $(k-1)$-form $\Gamma$ on $\widetilde X$ such that
\[
p^{*}\alpha=d\Gamma\quad\text{and}\quad\|\Gamma\|_{L^{\infty}(\widetilde X,\widetilde g)}<\infty.
\]
\end{definition}

Because any two metrics on a compact manifold are uniformly equivalent, the boundedness
condition is independent of the chosen base metric. The following elementary lemma allows us to work with de Rham classes.

\begin{lemma}\label{lem:cohomological}
Let $\alpha$ and $\beta$ be closed $k$-forms on a compact manifold $X$ with $[\alpha]_{\DR}=[\beta]_{\DR}$. Then $\alpha$ is $\widetilde d$-bounded if and only if
$\beta$ is $\widetilde d$-bounded.
\end{lemma}

\begin{proof}
Write $\beta-\alpha=d\eta$ for a smooth $(k-1)$-form $\eta$ on $X$. If $p^{*}\alpha=d\Gamma$ with $\Gamma$ bounded, then
\[
p^{*}\beta=p^{*}\alpha+d(p^{*}\eta)=d\bigl(\Gamma+p^{*}\eta\bigr).
\]
The form $\eta$ has bounded pointwise norm on the compact base, and $p$ is a local
isometry for lifted metrics; hence $p^{*}\eta$ is bounded on $\widetilde X$. Thus
$\Gamma+p^{*}\eta$ is a bounded primitive of $p^{*}\beta$. Interchanging $\alpha$ and
$\beta$ proves the converse.
\end{proof}

It follows that
\[
 \Vhyp^{k}(X):=\bigl\{c\in H^{k}_{\DR}(X,\R):c\text{ has a $\widetilde d$-bounded representative}\bigr\}
\]
is a well-defined real vector subspace of $H^{k}_{\DR}(X,\R)$. We set $\Vhyp^0(X):=0$.

The following ideal property is known for hyperbolic cohomology classes; see \cite[Remark~4.5]{BKS24}. We include the short differential-form proof for completeness and to fix the notation used later.

\begin{proposition}\label{prop:graded-ideal}
The graded subspace
\[
\Vhyp^{*}(X):=\bigoplus_{k\geq0}\Vhyp^{k}(X)\subset H^{*}_{\DR}(X,\R)
\]
is an ideal of the de Rham cohomology ring. More precisely, for all $r,s$, the cup product satisfies
\[
\Vhyp^{r}(X)\smile H^{s}_{\DR}(X,\R)\subseteq \Vhyp^{r+s}(X).
\]
\end{proposition}

\begin{proof}
Let $[\alpha]\in\Vhyp^{r}(X)$ and $[\beta]\in H^{s}_{\DR}(X,\R)$, with $d\alpha=d\beta=0$. Choose a bounded form $\Gamma$ on $\widetilde X$ such that $p^{*}\alpha=d\Gamma$. Since $p^{*}\beta$ is closed,
\[
p^{*}(\alpha\wedge\beta)=d\bigl(\Gamma\wedge p^{*}\beta\bigr).
\]
The form $p^{*}\beta$ is bounded because it is lifted from the compact base. Hence $\Gamma\wedge p^{*}\beta$ is bounded, proving that $[\alpha]\smile[\beta]\in\Vhyp^{r+s}(X)$.
\end{proof}

\subsection{Balanced cones and the fixed hyperbolic subspace}

A real $(n-1,n-1)$-form $\Omega$ is called strictly positive if it lies in the interior of the cone of strongly positive forms. Michelsohn proved that the map
\[
 \omega\longmapsto\omega^{n-1}
\]
is a bijection from positive $(1,1)$-forms to strictly positive $(n-1,n-1)$-forms \cite{Mic82}. Consequently, a compact complex $n$-fold is balanced if and only if it carries a strictly positive, $d$-closed $(n-1,n-1)$-form. We shall also refer to such an $(n-1,n-1)$-form as a balanced metric, and use the same convention in the hyperbolic setting.

After applying Ehresmann's fibration theorem to a proper holomorphic submersion and
shrinking the base, we may identify all nearby fibres with a fixed smooth manifold $X$
and write $J_t$ for their complex structures. The universal cover $p:\widetilde X\to X$ is then fixed as a smooth covering.

For each $t$, set
\[
\cB_t:=\left\{[\Omega]_{\DR}\in H^{2n-2}_{\DR}(X,\R):\begin{array}{l}
 d\Omega=0,\quad \Omega\in A^{n-1,n-1}_{J_t}(X,\R),\\
 \Omega\text{ is strictly positive}
 \end{array}
 \right\}.
\]
This is the image in de Rham cohomology of the balanced cone of $X_t$.

\begin{proposition}\label{prop:intersection}
For every $t$,
\[
X_t\text{ is balanced}\quad\Longleftrightarrow\quad\cB_t\neq\varnothing,
\]
and
\[
X_t\text{ is balanced hyperbolic}\quad\Longleftrightarrow\quad\cB_t\cap\Vhyp^{2n-2}(X)\neq\varnothing.
\]
\end{proposition}

\begin{proof}
The first equivalence is Michelsohn's root theorem. For the second, if $\Omega=\omega^{n-1}$ is positive and closed, then $\omega$ is balanced hyperbolic precisely when $[\Omega]_{\DR}$ belongs to $\Vhyp^{2n-2}(X)$. By \cref{lem:cohomological}, this depends only on the de Rham class.
\end{proof}

Thus balancedness of the nearby fibres ensures only that $\cB_t\neq\varnothing$; it does not guarantee an intersection with the fixed subspace $\Vhyp^{2n-2}(X)$. In particular, if a fixed class $c\in\Vhyp^{2n-2}(X)$ admits a strictly positive closed $(n-1,n-1)$-representative for every sufficiently small $J_t$, then \cref{prop:intersection} implies that each nearby $X_t$ is balanced hyperbolic.

\section{Counterexamples to deformation stability}
\label{sec:counterexample}

\subsection{The Iwasawa deformation}

Let $I$ be the Iwasawa threefold. It has global $(1,0)$-forms $\varphi^1,\varphi^2,\varphi^3$ satisfying
\[
d\varphi^1=d\varphi^2=0\qquad\textrm{and}\qquad d\varphi^3=-\varphi^1\wedge\varphi^2.
\]
The invariant Hermitian form
\[
\alpha=\frac{i}{2}\sum_{j=1}^{3}\varphi^j\wedge\overline{\varphi^j}
\]
is balanced. Indeed, writing $e_j=\frac{i}{2}\varphi^j\wedge\overline{\varphi^j}$, one has $\alpha^2=2(e_1e_2+e_1e_3+e_2e_3)$. The term $e_1e_2$ is closed. Whenever $d$ acts on
$\varphi^3$ or $\overline{\varphi^3}$ in $e_1e_3$ or $e_2e_3$, a repeated $\varphi^1,\varphi^2,\overline{\varphi^1}$, or $\overline{\varphi^2}$ occurs; hence all such terms vanish and $d\alpha^2=0$.

The following classical input is due to Alessandrini--Bassanelli, using Nakamura's
small deformations of the Iwasawa manifold \cite{AB90}.

\begin{theorem}\label{thm:AB}
There exists a one-parameter analytic family $\{I_t\}_{t\in\Delta}$ of compact complex
threefolds such that $I_0=I$ and $I_t$ is not balanced for every sufficiently small
$t\neq 0$.
\end{theorem}

This family is the standard counterexample to the openness of balancedness; see \cite{FY11}. Only the conclusion of \cref{thm:AB}, not the detailed structure equations, is needed below.

\subsection{A K\"ahler hyperbolic stabilizing factor}

\begin{lemma}\label{lem:curves-KH}
Let $C$ be a compact Riemann surface of genus at least two. Then it is K\"ahler hyperbolic. Consequently, any product $Y=C_1\times\cdots\times C_m$ of such curves is K\"ahler hyperbolic.
\end{lemma}

\begin{proof}
The universal cover of $C$ is the unit disc $\D$. Up to an inessential positive
normalization, its Poincar\'e form is
\[
 \omega_{\D}=\frac{i\,dz\wedge d\bar z}{(1-|z|^2)^2}.
\]
Set
\[
\theta=\frac{i}{2}\,\frac{z\,d\bar z-\bar z\,dz}{1-|z|^2}.
\]
A direct differentiation gives $d\theta=\omega_{\D}$. The coefficients of $\theta$
grow like $(1-|z|^2)^{-1}$, while the dual Poincar\'e metric contracts covectors by a
factor comparable to $1-|z|^2$; hence $|\theta|_{\omega_{\D}}$ is bounded (in fact it
is bounded by a universal constant). Since $\omega_{\D}$ is invariant under the deck group, it descends to a K\"{a}hler hyperbolic form on $C$.

For a product, let $\kappa_j$ be the hyperbolic form on $C_j$ and let $\theta_j$ be the
bounded primitive of its lift. Then
\[
\kappa=\sum_{j=1}^{m}\operatorname{pr}_j^{*}\kappa_j\qquad\textrm{and}\qquad\theta=\sum_{j=1}^{m}\operatorname{pr}_j^{*}\theta_j
\]
satisfy $\widetilde\kappa=d\theta$, and $\theta$ is bounded for the lifted product
metric.
\end{proof}

The next proposition is a stabilization principle. The corresponding product result in \cite[Proposition~2.17]{MP23} assumes that the other factor is K\"ahler; here it need only be balanced. We also construct the bounded primitive explicitly.

\begin{proposition}\label{prop:stabilization}
Let $(X^n,\alpha)$ be a compact balanced manifold and let $(Y^m,\kappa)$ be a compact
K\"ahler hyperbolic manifold with $m\geq2$. Then $X\times Y$ is balanced hyperbolic.
More precisely, $\omega=\operatorname{pr}_X^{*}\alpha+\operatorname{pr}_Y^{*}\kappa$ is a balanced hyperbolic metric.
\end{proposition}

\begin{proof}
Suppress pull-back symbols and put $N=n+m$. Since powers above the dimensions of the
factors vanish, only two terms survive in the binomial expansion:
\begin{equation}\label{eq:product-power}
\omega^{N-1}=\binom{N-1}{n-1}\alpha^{n-1}\wedge\kappa^{m}+\binom{N-1}{n}\alpha^{n}\wedge\kappa^{m-1}.
\end{equation}
The first term is closed because $d\alpha^{n-1}=0$ and $d\kappa=0$; the second is closed because $\alpha^n$ has top degree on $X$ and $d\kappa=0$. Thus $\omega$ is balanced.

On the universal cover, choose a bounded one-form $\theta$ with $\widetilde\kappa=d\theta$. Define
\begin{align*}
\Gamma={}&\binom{N-1}{n-1}\widetilde\alpha^{\,n-1}\wedge\theta\wedge\widetilde\kappa^{\,m-1}+\binom{N-1}{n}\widetilde\alpha^{\,n}\wedge\theta\wedge\widetilde\kappa^{\,m-2}.
\end{align*}
Since $d\widetilde\alpha^{n-1}=d\widetilde\alpha^n=d\widetilde\kappa=0$, differentiation
and \eqref{eq:product-power} give
\[
 d\Gamma=\widetilde\omega^{N-1}.
\]
All lifted forms coming from the compact factors have uniformly bounded norm, and $\theta$ is bounded; hence $\Gamma$ is bounded. Therefore $\omega$ is balanced hyperbolic.
\end{proof}

\subsection{Descent of balancedness}

\begin{lemma}\label{lem:factor-balanced}
Let $X^n$ and $Y^m$ be compact complex manifolds. If $X\times Y$ is balanced, then $X$ is balanced. The same conclusion holds with the two factors interchanged.
\end{lemma}

\begin{proof}
Let $q:X\times Y\to X$ be the projection and let $h$ be a balanced metric on the product. Set
\[
 \Xi=h^{n+m-1}.
\]
Then $\Xi$ is a strictly positive, $d$-closed form of bidegree $(n+m-1,n+m-1)$. Integration over the $m$-dimensional fibres defines
\[
 \Psi:=q_{*}\Xi\in A^{n-1,n-1}(X).
\]
Fibre integration commutes with $d$ up to the standard degree sign, so $d\Psi=0$.

It remains to prove strict positivity. Fix $x\in X$ and a non-zero $(1,0)$-covector $\lambda\in T^{*1,0}_xX$. On the fibre $\{x\}\times Y$,
\[
\Psi\wedge i\lambda\wedge\bar\lambda=q_{*}\bigl(\Xi\wedge q^{*}(i\lambda\wedge\bar\lambda)\bigr).
\]
The integrand is a positive top-degree form, and it is strictly positive because $\Xi$ is strictly positive and $\lambda\neq0$. Its integral over the compact fibre is therefore positive. This is the dual characterization of strict positivity for an $(n-1,n-1)$-form.

By Michelsohn's root theorem, there is a unique positive $(1,1)$-form $\beta$ on $X$ with $\beta^{n-1}=\Psi$. Since $d\Psi=0$, the metric $\beta$ is balanced.
\end{proof}

\subsection{Assembly of the counterexample}

Firstly, we show that the K\"{a}hler hyperbolicity is in general not closed under deformation.
\begin{proof}[Proof of \cref{thm:kahler-not-closed}]
Let $C_1,C_2,C_3$ be compact curves of genus at least two, and set
\[
Y=C_1\times C_2\times C_3.
\]
The product $Y$ is K\"ahler hyperbolic by \cref{lem:curves-KH}. Every smooth complex submanifold $D\subset Y$ is K\"ahler hyperbolic with the restricted product K\"ahler form. Indeed, the inclusion lifts from the universal cover of $D$ to that of $Y$, and the pull-back of a bounded primitive is bounded with respect to the induced metric on $D$.

Choose a sufficiently positive line bundle $L$ on the smooth projective threefold $Y$. By prescribing a nondegenerate quadratic jet at a point $p\in Y$ and applying Bertini away from $p$, we may choose a section $\sigma_0\in H^0(Y,L)$ whose zero divisor $D_0$ has precisely one ordinary double point at $p$ and is smooth elsewhere. Choose $\sigma_1$ with $\sigma_1(p)\neq0$. After shrinking the disc, the divisors
\[
D_t=\{\sigma_0+t\sigma_1=0\}\subset Y
\]
are smooth for $t\neq0$. They are K\"ahler hyperbolic by the preceding observation. By the holomorphic Morse lemma with parameters, near $p$ the family has local equation $xy=z^2-t$.

Make the base change $t=s^2$. The resulting local equation is
\begin{equation}\label{eq:node-base-change}
xy=(z-s)(z+s).
\end{equation}
Following the simultaneous resolution of an ordinary double point \cite{Ati58}, adjoin $[u:v]\in\mathbb P^1$ and impose
\begin{equation}\label{eq:node-resolution}
xv=(z-s)u,\qquad (z+s)v=yu.
\end{equation}
The two affine charts show that the resolved total space is smooth and that its projection to the $s$-disc is a submersion. On $u=1$, the equations become
\[
z=s+xv,\qquad y=xv^2+2sv,
\]
with coordinates $(x,v,s)$. On $v=1$, they become
\[
z=yu-s,\qquad x=yu^2-2su,
\]
with coordinates $(y,u,s)$. In both charts the structure map is projection to $s$. This construction is an isomorphism away from the singular point of the total space of \eqref{eq:node-base-change}, so it glues with the unchanged family away from $p$. The resulting proper holomorphic submersion $q:\mathcal Z\to\Delta_s$ has fibres
\[
Z_s\simeq D_{s^2}\quad(s\neq0),\qquad Z_0\simeq\widetilde D_0,
\]
where $\widetilde D_0\to D_0$ is the minimal resolution. Its exceptional curve is $E\simeq\mathbb P^1$. All fibres are smooth and projective.

If $Z_0$ admitted a K\"ahler-hyperbolic form $\omega$, then the map $i:E\hookrightarrow Z_0$ would lift to the universal cover $\widetilde Z_0$, since $E$ is simply connected. Writing $p^*\omega=d\eta$ gives
\[
0<\int_Ei^*\omega=\int_Ed(\widetilde i^*\eta)=0,
\]
contradicting positivity. Hence $Z_0$ is not K\"ahler hyperbolic, whereas all $Z_s$ for $s\ne0$ are K\"ahler hyperbolic.
\end{proof}

The same phenomenon holds for the balanced deformation openness. 
\begin{theorem}\label{thm:main-counterexample}
For every integer $N\geq5$, there is a one-parameter analytic family of compact complex $N$-folds $\{Z_t\}_{t\in\Delta}$ with $Z_0$ balanced hyperbolic but $Z_t$ not balanced for all sufficiently small $t\ne0$.
\end{theorem}

\begin{proof}
Fix $N\geq5$ and put $m=N-3\geq2$. Let $\{I_t\}$ be the Alessandrini--Bassanelli family from \cref{thm:AB}, and choose compact curves $C_1,\dots,C_m$ of genus at least two. Set
\[
Y=C_1\times\cdots\times C_m\qquad\textrm{and}\qquad Z_t=I_t\times Y.
\]
This is an analytic family of compact complex $N$-folds.

The central fibre $I_0$ is balanced, as shown above, while $Y$ is K\"ahler hyperbolic by
\cref{lem:curves-KH}. Since $m\geq2$, \cref{prop:stabilization} shows that $Z_0=I_0\times Y$ is balanced hyperbolic.

Suppose $Z_t$ were balanced for some sufficiently small $t\neq0$. Applying
\cref{lem:factor-balanced} to the projection $Z_t\to I_t$ would imply that $I_t$ is
balanced, contradicting \cref{thm:AB}. Thus every sufficiently small noncentral $Z_t$ fails to be balanced, and therefore cannot be balanced hyperbolic.
\end{proof}

\section{Deformation openness of K\"ahler hyperbolicity}\label{sec:monodromy}

We prove \cref{thm:kahler-openness}. The difficulty is that the flat hyperbolic subspace $\Vhyp^2$ need not respect the Hodge decomposition in a general analytic family. We first work over a global algebraic parameter space, where Deligne's fixed-part theorem controls flat endomorphisms, and then return to the given analytic deformation. The order of these steps is essential.

\subsection{Pure-type intersections of flat subsystems}\label{subsec:deligne-lemma}

\begin{lemma}\label{lem:flat-hodge-rank}
Let $f:\mathcal Y\to S$ be a smooth projective morphism with connected fibres over a smooth connected complex quasi-projective variety. Write
\[
\mathbb H_{\R}=R^k f_*\R
\]
for the real cohomology local system. For any real sub-local system $\mathbb V_{\R}\subset\mathbb H_{\R}$ and any $p+q=k$, the function
\[
s\longmapsto\dim_{\C}\bigl(V_{\C,s}\cap H^{p,q}(Y_s)\bigr)
\]
is locally constant, hence constant on $S$. Neither rationality of $\mathbb V_{\R}$ nor its compatibility with the Hodge decomposition is assumed.
\end{lemma}

\begin{proof}
The geometric cohomology local system is semisimple by Deligne's theorem \cite[Th\'eor\`eme~4.2.6]{Del71}, applied to the algebraic family of polarized Hodge structures. After extension of scalars to $\R$, the local system remains semisimple. Hence there exists a flat real projector
\[
P\in H^0\bigl(S,\operatorname{End}(\mathbb H_{\R})\bigr),\qquad P^2=P,\qquad \operatorname{im}P=\mathbb V_{\R}.
\]
The algebra
\[
A_{\Q}=H^0\bigl(S,\operatorname{End}(R^kf_*\Q)\bigr)
\]
of flat endomorphisms has, by Deligne's fixed-part theorem \cite[Corollaire~4.2.8(i)]{Del71}, a weight-zero Hodge structure such that its action on $H^k(Y_s,\Q)$ is a morphism of Hodge structures for every $s$. Since $A_{\Q}\otimes\R$ is the space of flat real endomorphisms, the complexification of $P$ decomposes uniquely as
\begin{equation}\label{eq:projector-hodge-decomp}
P=\sum_rP_r,\qquad P_r\in A_{\C}^{r,-r}.
\end{equation}
In particular, each $P_r$ is flat and
\[
P_r(H^{p,q}(Y_s))\subset H^{p+r,q-r}(Y_s)
\]
for every $s\in S$.

Define the following complex sub-local system:
\begin{equation}\label{eq:hodge-kernel-W}
\mathbb W_{\C}:=\ker(P_0-I)\cap\bigcap_{r\ne0}\ker(P_r)\subset\mathbb H_{\C}.
\end{equation}
Because the defining endomorphisms are flat, so is this subsystem. It is also stable under the Hodge decomposition: for $w=\sum_{p+q=k}w^{p,q}$ the summands of $P_rw=\sum_{p,q}P_rw^{p,q}$ have distinct Hodge types for fixed $r$, so $P_rw=0$ implies $P_rw^{p,q}=0$ for each $(p,q)$; the same applies to $P_0-I$. Since $P$ is real and $\overline{P_r}=P_{-r}$, the subsystem $\mathbb W_{\C}$ is conjugation invariant.

For a vector $v$ of pure type $(p,q)$,
\[
v\in V_{\C,s}\iff Pv=v\iff (P_0-I)v=0,\ P_rv=0\text{ for }r\ne0\iff v\in\mathbb W_{\C,s}.
\]
Thus
\begin{equation}\label{eq:flat-type-equality}
V_{\C,s}\cap H^{p,q}(Y_s)=\mathbb W_{\C,s}\cap H^{p,q}(Y_s).
\end{equation}
The Hodge projection $\operatorname{pr}^{p,q}_s$ restricts to a continuous family of idempotent operators on the flat vector bundle $\mathbb W_{\C}$. Its rank is its trace, a continuous integer-valued function of $s$; hence the rank is locally constant. By \eqref{eq:flat-type-equality} this is exactly the required dimension.
\end{proof}

\begin{remark}\label{rem:why-algebraic}
The algebraic-base hypothesis is essential to this argument. On a simply connected analytic disc every underlying cohomology local system is trivial, yet the Hodge decomposition of a nontrivial analytic variation may move. It would therefore be incorrect to apply the fixed-part statement to arbitrary constant sections of the restricted endomorphism local system on that disc. The flat Hodge components in \eqref{eq:projector-hodge-decomp} are obtained over the global algebraic base, then used locally.
\end{remark}

\subsection{Hyperbolic classes in algebraic families}\label{subsec:algebraic-kahler}

\begin{proposition}\label{prop:hyperbolic-flat}
For a smooth projective family $f:\mathcal Y\to S$ as in \cref{lem:flat-hodge-rank}, the spaces $\Vhyp^k(Y_s)$ form a real sub-local system of $R^kf_*\R$. In particular, for $k=2$ the function
\begin{equation}\label{eq:algebraic-rank}
s\longmapsto\dim_{\R}\bigl(\Vhyp^2(Y_s)\cap H^{1,1}(Y_s,\R)\bigr)
\end{equation}
is locally constant.
\end{proposition}

\begin{proof}
By \cref{lem:cohomological}, $\Vhyp^k(M)$ depends only on the de Rham class. It is also natural under diffeomorphisms $M\to M'$. Indeed, the diffeomorphism lifts to one of their universal covers, and its derivative and inverse derivative are bounded with respect to fixed metrics on the compact bases. Therefore a lifted bounded primitive pulls back to a bounded primitive. Since the monodromy of a smooth proper family acts by diffeomorphisms, it preserves $\Vhyp^k$ and defines the indicated sub-local system.

Apply \cref{lem:flat-hodge-rank} in degree two to $\mathbb V_{\R}=\Vhyp^2$. For any complex subspace $V_{\C}$ defined by a real subspace $V_{\R}$, complex conjugation gives
\[
(V_{\R}\cap H^{1,1}(Y_s,\R))\otimes_{\R}\C=V_{\C,s}\cap H^{1,1}(Y_s).
\]
Hence the constant complex dimension given by the lemma equals the real dimension in \eqref{eq:algebraic-rank}.
\end{proof}

\paragraph{Continuing hyperbolic K\"ahler classes over an algebraic base.}
We first prove openness of the K\"ahler hyperbolic locus in the analytic topology of $f:\mathcal Y\to S$; the construction will also be used in the Hilbert-scheme argument.

Fix a K\"ahler-hyperbolic fibre $Y_{s_0}$ and a corresponding K\"ahler form $\omega_0$. In a sufficiently small contractible analytic neighbourhood of $s_0$, use Ehresmann's theorem to identify the fibres differentiably with $M=Y_{s_0}$, and write $J_t$ for their complex structures. The vector space $V=\Vhyp^2(M)$ is fixed. The real Hodge subspaces
\[
P_t:=H^{1,1}(Y_t,\R)\subset H^2_{\DR}(M,\R)
\]
form a smooth vector subbundle, since the family is smooth projective and its degree-two Hodge numbers are constant. By \cref{prop:hyperbolic-flat}, the dimension of $V\cap P_t$ is locally constant. Consequently these intersections form a smooth vector subbundle of $V$; the class $c_0=[\omega_0]_{\DR}$ extends to a smooth family $c_t\in V\cap P_t$ with $c_t\to c_0$.

Choose a real basis $v_1=c_0,v_2,\ldots,v_m$ of $V$. Let $\rho_1=\omega_0,\rho_2,\ldots,\rho_m$ be closed representatives, and choose bounded one-forms $\eta_j$ on the universal cover $p:\widetilde M\to M$ such that $p^*\rho_j=d\eta_j$. Write $c_t=\sum_j a_j(t)v_j$ with $a_1(0)=1$ and $a_j(0)=0$ for $j>1$, and put
\[
\Omega_t=\sum_{j=1}^m a_j(t)\rho_j,\qquad A_t=(\Omega_t)^{0,2}_{J_t}.
\]
Then $\Omega_t$ is real and closed, represents $c_t$, and converges to $\omega_0$ in $C^\infty$. Its $(0,2)$ component $A_t$ is $\bar\partial_t$-closed, has zero Dolbeault cohomology class because $c_t$ has type $(1,1)$, and converges to zero in $C^\infty$.

Let $G_t$ be the $\bar\partial_t$-Green operator on $(0,2)$-forms for a smooth family of Hermitian metrics, and set
\[
u_t=\bar\partial_t^*G_tA_t.
\]
The Hodge numbers are constant in the projective family, so parametric elliptic theory gives $u_t\to0$ in $C^\infty$ and $\bar\partial_tu_t=A_t$. It follows that
\[
\omega_t=\Omega_t-d(u_t+\overline{u_t})
\]
is real, closed, and of type $(1,1)$ on $Y_t$. Since $\omega_t\to\omega_0$ in $C^\infty$, it is positive for small $t$. In addition, its lift has the explicitly bounded primitive
\[
\eta_t=\sum_{j=1}^m a_j(t)\eta_j-p^*(u_t+\overline{u_t}),\qquad d\eta_t=p^*\omega_t.
\]
The first summand is bounded because the coefficients $a_j(t)$ stay bounded; the second is bounded because it is pulled back from the compact base. The metrics $\omega_t(\cdot,J_t\cdot)$ are uniformly equivalent to the initial K\"ahler metric on $M$, so the boundedness persists for the lifted fibrewise metrics. Hence all sufficiently near fibres are K\"ahler hyperbolic.

\subsection{The Hilbert locus and descent}\label{subsec:hilbert-descent}

\begin{proposition}\label{prop:hilbert-KH-open}
Fix $N$ and a Hilbert polynomial $P$, and let $B$ be the reduced open subscheme of $\operatorname{Hilb}^P(\mathbb P^N)$ parametrizing smooth connected complex varieties. Then
\[
B_{\mathrm{KH}}:=\{b\in B(\C):\mathcal U_b\text{ is K\"ahler hyperbolic}\}
\]
is open in $B^{\mathrm{an}}$, where $\mathcal U\to B$ is the universal smooth projective family. Moreover, the intrinsic function $b\mapsto\dim_{\R}(\Vhyp^2(\mathcal U_b)\cap H^{1,1}(\mathcal U_b,\R))$ is locally constant on $B^{\mathrm{an}}$.
\end{proposition}

\begin{proof}
Resolve the finitely many irreducible components of $B$ separately to obtain a proper surjective morphism
\[
r:S=\coprod_iS_i\longrightarrow B
\]
with each $S_i$ smooth quasi-projective. The base-changed universal family over each $S_i$ is smooth and projective. The continuation argument above shows that $r^{-1}(B_{\mathrm{KH}})$ is open in $S^{\mathrm{an}}$. Let $\rho(b)$ denote the rank function in the statement. By \cref{prop:hyperbolic-flat}, $\rho\circ r$ is locally constant on $S$. For each integer $j$, the inverse image of the level set $\{b:\rho(b)=j\}$ is both open and closed in $S$; properness and surjectivity of $r$ imply the same for the level set in $B$. Proper maps of complex analytic spaces are closed, and the definition of $B_{\mathrm{KH}}$ depends only on the isomorphism class of a fibre; hence
\[
B\setminus B_{\mathrm{KH}}=r\bigl(S\setminus r^{-1}(B_{\mathrm{KH}})\bigr)
\]
is closed. This proves both assertions without requiring the Hilbert scheme or its smooth-fibre locus to be nonsingular.
\end{proof}

\begin{proof}[Proof of \cref{thm:kahler-openness}]
Since $X_0$ is compact K\"ahler hyperbolic, its canonical bundle $K_{X_0}$ is ample by \cite[Theorem~2.11]{CY18}. The relative canonical bundle $K_{\mathcal X/\Delta}$ restricts to $K_{X_0}$. Choose $m\gg0$ so that $K_{X_0}^{\otimes m}$ is very ample and has vanishing higher coherent cohomology. Grauert's base-change theorem and the openness of the embedding condition give, after shrinking $\Delta$, a closed embedding over $\Delta$
\begin{equation}\label{eq:relative-canonical-embedding}
\mathcal X\hookrightarrow\mathbb P^N\times\Delta,
\end{equation}
induced by $K_{\mathcal X/\Delta}^{\otimes m}$. We use here the analytic relative-embedding statements of \cite{GCartVIII}; the argument is standard because the restriction map on sections and separation of points and tangent directions hold on the central fibre.

The fibres of \eqref{eq:relative-canonical-embedding} have a common Hilbert polynomial $P$. By Grothendieck's analytic Hilbert theorem \cite[Th\'eor\`eme~1.1]{GCartIX}, the embedded analytic family is classified by a holomorphic map
\[
h:\Delta\longrightarrow B^{\mathrm{an}},
\]
where $B$ is the smooth connected-fibre locus of the reduced algebraic Hilbert scheme from \cref{prop:hilbert-KH-open}. The map $h$ need not be algebraic. Since $h(0)\in B_{\mathrm{KH}}$ and $B_{\mathrm{KH}}$ is open, $h(t)\in B_{\mathrm{KH}}$ for all sufficiently small $t$. The pullback family is the given deformation, proving the assertion.
\end{proof}

\section{Aeppli cohomology and balanced hyperbolicity}
\label{sec:compact-program}

Let $X_t=(X,J_t)$ be a small holomorphic deformation, and set
\[
H=H^{2n-2}_{\DR}(X,\R),\qquad V=\Vhyp^{2n-2}(X).
\]
Let $\mathcal P_t\subset H$ be the space of classes admitting real, closed $(n-1,n-1)$-representatives on $X_t$. We identify this space as the kernel of an Aeppli obstruction map, compute its first variation, and use a constant-rank hypothesis to continue positive representatives. The resulting forms retain bounded primitives on the universal cover.

\subsection{The Aeppli off-type realization}
\label{subsec:aeppli-obstruction}

Recall that the Aeppli cohomology of a complex manifold is
\[
H^{p,q}_{\Aep}(X)=\frac{\kerop(\partial\dbar:A^{p,q}\to A^{p+1,q+1})}{\im\partial+\im\dbar}.
\]
Fix a real de Rham class $c\in H$ and a real $d$-closed representative $\Omega$. Relative to $J_t$, decompose the real $(2n-2)$-form $\Omega$ uniquely as
\[
\Omega=A_t+B_t+\overline{A_t},\qquad A_t\in A^{n-2,n}_{J_t}(X),\quad B_t\in A^{n-1,n-1}_{J_t}(X,\R).
\]
Since $A_t$ has antiholomorphic degree $n$, one has $\dbar_tA_t=0$ by bidegree, and
hence $\partial_t\dbar_tA_t=0$. Thus $A_t$ defines an Aeppli class. We set
\[
 \mathfrak o_t(c):=[A_t]_{\Aep}\in H^{n-2,n}_{\Aep}(X_t)
\]
and call $\mathfrak o_t(c)$ the \emph{off-type Aeppli obstruction} of $c$ at $t$.

\begin{proposition}\label{prop:obstruction-well-defined}
The class $\mathfrak o_t(c)$ is independent of the chosen $d$-closed representative
$\Omega$ of $c$. Moreover, $\mathfrak o_t(c)=0$ if and only if $c$ admits a real, $d$-closed representative of pure type $(n-1,n-1)$ for $J_t$.
\end{proposition}

\begin{proof}
If $\Omega'=\Omega+d\eta$, then the $(n-2,n)$-component changes by
\[
(d\eta)^{n-2,n}=\partial_t\eta^{n-3,n}+\dbar_t\eta^{n-2,n-1},
\]
which is Aeppli-exact. Hence $\mathfrak o_t(c)$ is well-defined.

Suppose first that $\mathfrak o_t(c)=0$. Then there exist
\[
v_t\in A^{n-3,n}_{J_t}(X),\qquad u_t\in A^{n-2,n-1}_{J_t}(X)
\]
such that
\begin{equation}\label{eq:aeppli-primitive}
 A_t=\partial_tv_t+\dbar_tu_t.
\end{equation}
Set
\[
\eta_t=v_t+u_t+\overline{u_t}+\overline{v_t},\qquad\Omega_t'=\Omega-d\eta_t.
\]
The $(n-2,n)$-component of $d\eta_t$ is $A_t$, and the $(n,n-2)$-component is $\overline{A_t}$. Therefore $\Omega_t'$ is real and of pure type $(n-1,n-1)$. It is $d$-closed and represents $c$.

Conversely, suppose that a real $d$-closed $(n-1,n-1)$-form $\Omega_t'$ represents $c$. Writing $\Omega-\Omega_t'=d\eta$ and taking the $(n-2,n)$-component gives
\[
 A_t=\partial_t\eta^{n-3,n}+\dbar_t\eta^{n-2,n-1},
\]
so $[A_t]_{\Aep}=0$.
\end{proof}

Accordingly, if
\[
\mathcal P_t:=\{c\in H:c\text{ admits a real }d\text{-closed }(n-1,n-1)\text{-representative on }X_t\},
\]
then
\[
 \mathcal P_t=\ker\mathfrak o_t.
\]
The pure-type condition is therefore equivalent to the vanishing of a finite-dimensional cohomological obstruction.

Together with openness of strict positivity, \cref{prop:obstruction-well-defined} gives a useful sufficient condition. If a class $c\in V$ has a strictly positive closed representative $\Omega$ on $X_0$, and the off-type component of $\Omega$ on $X_t$ admits corrections $A_t=\partial_tv_t+\dbar_tu_t$ with $u_t,v_t\to0$ in $C^1$, then $\Omega-d(v_t+u_t+\bar u_t+\bar v_t)$ remains closed, strictly positive, and of type $(n-1,n-1)$. By \cref{prop:intersection}, $X_t$ is balanced hyperbolic. The uniform smallness needed here is supplied below by parametric Aeppli Hodge theory.

We next use Aeppli--Hodge theory to construct the harmonic projections and homotopy operators required for both the first-variation calculation and the continuation argument.

On a compact Hermitian manifold there is a fourth-order self-adjoint elliptic Aeppli
Laplacian $\Delta_{\Aep}$ whose kernel $\mathcal H^{p,q}_{\Aep}$ represents $H^{p,q}_{\Aep}$, and one has the orthogonal decomposition
\[
A^{p,q}=\mathcal H^{p,q}_{\Aep}\oplus(\operatorname{Im}\partial+\operatorname{Im}\dbar)\oplus\operatorname{Im}(\partial\dbar)^*.
\]
This decomposition is proved in \cite[Section~2.c]{Sch07}; see also \cite{Pop15}. For a smooth family of complex structures and Hermitian metrics, standard parametric elliptic theory gives smoothly varying harmonic projections, Green operators, and homotopy operators whenever the relevant Aeppli number is locally constant \cite{HX25,KS60}. More explicitly, after shrinking the parameter disc, local constancy of $\dim\ker\Delta_{\Aep,t}$ makes the harmonic projections and Green operators depend smoothly on $t$; elliptic regularity then gives uniform Sobolev bounds on compact parameter subsets. The resulting Aeppli homotopy operators may therefore be chosen to depend smoothly on $t$ and to gain one Sobolev derivative. Thus, for every Sobolev index $s\geq0$, there are locally uniformly bounded operators
\[
(S_t^{\partial},S_t^{\dbar}):\ker(\partial_t\dbar_t)\cap H^s(A^{p,q}_{J_t})\longrightarrow H^{s+1}(A^{p-1,q}_{J_t})\oplus H^{s+1}(A^{p,q-1}_{J_t})
\]
such that
\begin{equation}\label{eq:aeppli-homotopy}
\alpha-\mathcal H_{\Aep,t}\alpha=\partial_tS_t^{\partial}\alpha+\dbar_tS_t^{\dbar}\alpha.
\end{equation}

\begin{lemma}\label{lem:parametric-aeppli-small-correction}
Suppose that $h^{n-2,n}_{\Aep}(X_t)$ is locally constant. After shrinking the parameter
disc, for every $s\geq0$ there is a constant $C_s>0$, independent of $t$, with the
following property. If
\[
A_t\in H^s(A^{n-2,n}_{J_t}),\qquad \partial_t\dbar_tA_t=0,\qquad\mathcal H_{\Aep,t}A_t=0,
\]
then there are
\[
v_t\in H^{s+1}(A^{n-3,n}_{J_t}),\qquad u_t\in H^{s+1}(A^{n-2,n-1}_{J_t})
\]
such that
\begin{equation}\label{eq:uniform-aeppli-estimate}
A_t=\partial_tv_t+\dbar_tu_t,\qquad\|v_t\|_{H^{s+1}}+\|u_t\|_{H^{s+1}}\leq C_s\|A_t\|_{H^s}.
\end{equation}
Consequently, if $A_t\to0$ in $C^{\infty}$, then $u_t,v_t\to0$ in $C^{\infty}$.
\end{lemma}

\begin{proof}
Take $v_t=S_t^{\partial}A_t$ and $u_t=S_t^{\dbar}A_t$ in \eqref{eq:aeppli-homotopy}. Smooth parameter dependence of these operators and compactness of a smaller closed parameter disc give the constant $C_s$ uniformly in $t$. Here the Sobolev norms are taken after a fixed smooth trivialization of the varying type bundles; all such choices are uniformly equivalent for small $t$. Applying \eqref{eq:uniform-aeppli-estimate} for every $s$ and then Sobolev embedding gives the final assertion.
\end{proof}

In particular, whenever $h^{n-2,n}_{\Aep}(X_t)$ is locally constant, the harmonic spaces
form a smooth vector bundle
\begin{equation}\label{eq:aeppli-bundle}
\mathcal E_{\Aep}:=\bigsqcup_{t\in\Delta}\mathcal H^{n-2,n}_{\Aep,t}\longrightarrow\Delta,
\end{equation}
which we identify fibrewise with $H^{n-2,n}_{\Aep}(X_t)$.

\subsection{First variation: the Aeppli deformation term}
\label{subsec:first-variation}

The first variation formula identifies the infinitesimal condition for a class in the fixed hyperbolic subspace $V$ to remain in the moving pure-type subspace $\mathcal P_t$.

Assume throughout this subsection that $h^{n-2,n}_{\Aep}(X_t)$ is locally constant, and
identify Aeppli classes with their harmonic representatives in the bundle $\mathcal E_{\Aep}$ from \eqref{eq:aeppli-bundle}. At a zero of a section, its covariant derivative is independent of the chosen connection.

Let $\varphi(t)=t\kappa+O(t^2)$ be a Beltrami differential for a real one-parameter path
$J_t$. The linearized Maurer--Cartan equation gives $\dbar\kappa=0$. We use the convention that, in local $J_0$-holomorphic coordinates, the $(1,0)$-cotangent bundle of $J_t$ is spanned to first order by
\[
 dz^i+\varphi^i_{\bar j}(t)\,d\bar z^j.
\]
With the opposite convention, the sign of the contraction term below is reversed; its
vanishing and range conditions are unchanged.

\begin{proposition}\label{prop:first-variation}
Let $c_t\in H$ be differentiable, and suppose that $c_0$ admits a real $d$-closed
$(n-1,n-1)$-representative $\Omega_0$. If $\kappa$ is the Kodaira--Spencer direction,
then
\begin{equation}\label{eq:first-variation}
\left.\nabla_{\partial/\partial t}\mathfrak o_t(c_t)\right|_{t=0}=\mathfrak o_0(\dot c_0)-[\iota_\kappa\Omega_0]_{\Aep}\in H^{n-2,n}_{\Aep}(X_0).
\end{equation}
Here the derivative is connection-independent because $\mathfrak o_0(c_0)=0$. Consequently, a fixed class has vanishing linearized obstruction exactly when
$[\iota_\kappa\Omega_0]_{\Aep}=0$; if $c_0\in V$, first-order continuation inside
$V\cap\mathcal P_t$ is possible exactly when
\begin{equation}\label{eq:first-order-V}
[\iota_\kappa\Omega_0]_{\Aep}\in\mathfrak o_0(V).
\end{equation}
When this holds, the admissible velocities are
\[
\mathcal S_{\kappa}(c_0)=\{\xi\in V:\mathfrak o_0(\xi)=[\iota_\kappa\Omega_0]_{\Aep}\}.
\]
\end{proposition}

\begin{proof}
Choose real $d$-closed representatives $\Omega_t=\Omega_0+t\Theta+o(t)$ of $c_t$, with $[\Theta]_{\DR}=\dot c_0$, and set $A_t=(\Omega_t)^{n-2,n}_{J_t}$. The moving-type
linearization is
\[
\dot A_0=\Theta^{n-2,n}_{J_0}-\iota_\kappa\Omega_0.
\]
Since $A_0=0$, differentiating the harmonic representative $\mathcal H_{\Aep,t}A_t$ gives $\mathcal H_{\Aep,0}\dot A_0$. Passing to Aeppli cohomology yields \eqref{eq:first-variation}. If $c_t\in V$ is required to satisfy $\mathfrak o_t(c_t)=0$ to first order, then $\dot c_0\in V$ and the displayed derivative must vanish. This gives precisely the stated range and affine-velocity conditions.
\end{proof}

In particular, if $c_0\in V\cap\mathcal P_0$ and $[\iota_\kappa\Omega_0]_{\Aep}\notin\mathfrak o_0(V)$, then no differentiable path $c_t\in V\cap\mathcal P_t$ can pass through $c_0$. This obstructs continuation of the specified class, but does not rule out a different positive hyperbolic class on nearby fibres.

\subsection{Continuation and positivity}
\label{subsec:constant-rank}

The obstruction restricted to $V$ is a smooth family of finite-dimensional linear maps. Constant rank allows us to continue the classes; Aeppli homotopy operators then restore positivity.

Continue to assume that $h^{n-2,n}_{\Aep}(X_t)$ is locally constant, and use the smooth
bundle $\mathcal E_{\Aep}$ from \eqref{eq:aeppli-bundle}. Fix a Riemannian metric on $X$ and let
\[
 s:H\longrightarrow Z^{2n-2}(X,\R)
\]
assign to each de Rham class its unique harmonic representative. Thus $s$ is real-linear and is a right inverse of the natural quotient map
\[
q:Z^{2n-2}(X,\R)\longrightarrow H,\qquad \Omega\longmapsto[\Omega]_{\DR}.
\]
Define a smooth real-linear bundle map
\begin{equation}\label{eq:obstruction-map}
\mathcal O_t:H\longrightarrow\mathcal E_{\Aep,t},\qquad\mathcal O_t(c):=\mathcal H_{\Aep,t}\bigl((s(c))^{n-2,n}_{J_t}\bigr).
\end{equation}
Thus $\mathcal O_t(c)$ is precisely the Aeppli-harmonic representative of $\mathfrak o_t(c)$.

By the representative-independence in \cref{prop:obstruction-well-defined},
\begin{equation}\label{eq:kernel-pure}
\ker\mathcal O_t=\mathcal P_t.
\end{equation}
This identity is independent of the representative section and of the auxiliary Hermitian metrics.

\begin{theorem}\label{thm:constant-rank}
Let $X_t=(X,J_t)$ be a small deformation and let $c_0\in\mathcal B_0\cap V$. Suppose that
\begin{enumerate}[label=\textup{(\alph*)}]
\item $h^{n-2,n}_{\Aep}(X_t)$ is locally constant, and
\item $\dim_{\R}(V\cap\mathcal P_t)$ is locally constant.
\end{enumerate}
Then $X_t$ is balanced hyperbolic for every sufficiently small $t$. More precisely, there exists a smooth path
\[
c_t\in V\cap\mathcal P_t,\qquad c_{t=0}=c_0,
\]
such that $c_t$ contains a balanced hyperbolic form on $X_t$. Equivalently, condition~\textup{(b)} may be replaced by local constancy of the real rank of $\mathcal O_t|_V$. Under condition~\textup{(a)}, it also suffices that $\mathcal O_0|_V$ be surjective as a real-linear map.
\end{theorem}

\begin{proof}
By \eqref{eq:kernel-pure}, the kernels of $\mathcal O_t|_V$ are precisely $V\cap\mathcal P_t$. Under condition~\textup{(a)}, the harmonic target spaces form a smooth vector bundle. Condition~\textup{(b)} therefore says that this family of finite-dimensional maps has constant rank. Its kernels form a smooth vector subbundle of $V\times\Delta$, so $c_0$ extends to a smooth path
\[
c_t\in V\cap\mathcal P_t,\qquad c_{t=0}=c_0.
\]
If $\mathcal O_0|_V$ is surjective, it remains surjective nearby and the same conclusion follows. We now construct positive representatives.

Retain the real-linear section $s$ used in \eqref{eq:obstruction-map}, and let $\Omega_0$ be a strictly positive, real, $d$-closed $(n-1,n-1)$-representative of $c_0$. Define
\[
\Omega_t:=\Omega_0+s(c_t-c_0),\qquad A_t:=(\Omega_t)^{n-2,n}_{J_t}.
\]
Then $\Omega_t$ is real and $d$-closed, represents $c_t$, and converges to $\Omega_0$ in $C^\infty$. Moreover,
\[
\Omega_t-s(c_t)=\Omega_0-s(c_0)
\]
is $d$-exact. Hence, by the representative-independence of the off-type Aeppli obstruction established in \cref{prop:obstruction-well-defined},
\[
\mathcal H_{\Aep,t}A_t=\mathcal O_t(c_t)=0.
\]
Since $A_0=0$, we have $A_t\to0$ in $C^\infty$. Applying
\cref{lem:parametric-aeppli-small-correction} gives, for every $s\geq0$,
\[
A_t=\partial_tv_t+\dbar_tu_t,\qquad\|v_t\|_{H^{s+1}}+\|u_t\|_{H^{s+1}}\leq C_s\|A_t\|_{H^s}.
\]
Thus $u_t,v_t\to0$ in $C^{\infty}$, as required for the positivity correction.
Consequently,
\[
\Omega_t'=\Omega_t-d\bigl(v_t+u_t+\overline{u_t}+\overline{v_t}\bigr)
\]
is real, $d$-closed, of type $(n-1,n-1)$ for $J_t$, and represents $c_t$. Moreover, $\Omega_t'\to\Omega_0$ in $C^{\infty}$. By openness of strict positivity, $\Omega_t'$ is strictly positive for small $t$. Since $c_t\in V$, \cref{prop:intersection} shows that the Michelsohn root of $\Omega_t'$ is a balanced hyperbolic metric on $X_t$.
\end{proof}

When the Aeppli obstruction space vanishes, no rank hypothesis is needed.
\begin{corollary}\label{thm:aeppli-vanishing}
Suppose $X_0$ is balanced hyperbolic and
\[
H^{n-2,n}_{\Aep}(X_0)=0.
\]
Then $X_t$ is balanced hyperbolic for all sufficiently small $t$. Equivalently, it suffices that $h^{2,0}_{\BC}(X_0)=0$.
\end{corollary}
\begin{proof}
Upper semicontinuity of the kernel dimension of the Aeppli Laplacian gives, for small $t$,
\[
H^{n-2,n}_{\Aep}(X_t)=0.
\]
Thus the target of $\mathcal O_t$ vanishes and $V\cap\mathcal P_t=V$. The constant-rank theorem applies. Bott--Chern/Aeppli duality gives the equivalent condition.
\end{proof}

If, in addition, the central fibre is a $\ddbar$-manifold, the criterion can be made more explicit.
\begin{corollary}\label{cor:ddbar-transverse}
Suppose that $X_0$ is a balanced hyperbolic $\ddbar$-manifold. With $\mathcal P_0$ as
above, set
\[
V^{\perp}:=\left\{\xi\in H^2_{\DR}(X,\R):\int_Xc\wedge\xi=0\text{ for all }c\in V\right\}
\]
and
\[
\mathcal T_0:=\left\{[\alpha+\bar\alpha]_{\DR}:\alpha\in A^{2,0}(X_0),\ d\alpha=0\right\}.
\]
Equivalently, $\mathcal T_0$ is the conjugation-fixed real locus of $H^{2,0}_{\BC}(X_0)\oplus H^{0,2}_{\BC}(X_0)$ inside $H^2_{\DR}(X,\C)$. Then the following are equivalent:
\begin{enumerate}[label=\textup{(\roman*)}]
\item $V+\mathcal P_0=H$;
\item $\mathcal O_0|_V:V\to H^{n-2,n}_{\Aep}(X_0)$ is surjective;
\item $V^{\perp}\cap\mathcal T_0=\{0\}$.
\end{enumerate}
Any of these conditions implies that every sufficiently small $X_t$ is balanced hyperbolic.
\end{corollary}

\begin{proof}
Openness of the $\ddbar$ property \cite{Wu06}, together with upper semicontinuity and
constancy of the Betti numbers, gives local constancy of the relevant Aeppli number. On $X_0$, the Bott--Chern-to-Aeppli map is an isomorphism, so $\mathcal O_0:H\to H^{n-2,n}_{\Aep}(X_0)$ is surjective with kernel $\mathcal P_0$. Surjectivity of the restriction
$\mathcal O_0|_V$ is equivalent to $V+\mathcal P_0=H$. This proves \textup{(i)}$\Leftrightarrow$\textup{(ii)}. Under the Hodge decomposition, $(\mathcal P_0)^\perp=\mathcal T_0$ for the Poincar\'e pairing, giving \textup{(i)}$\Leftrightarrow$\textup{(iii)}. The conclusion follows from the surjectivity case of \cref{thm:constant-rank}.
\end{proof}

\begin{remark}
\Cref{thm:aeppli-vanishing} is the rank-zero extreme of \cref{thm:constant-rank}: the entire Aeppli target vanishes, so $\mathcal P_t=H$ and $V\cap\mathcal P_t=V$. At the opposite extreme, the co-polarized situation treated in the literature (see \cref{subsec:compact-literature}) already provides a fixed section $c_t\equiv c_0$ of $V\cap\mathcal P_t$. \cref{thm:constant-rank} allows the class to drift, but only inside the fixed hyperbolic subspace $V$. A rank jump of $\mathcal O_t|_V$ may destroy the smooth kernel bundle used in the proof and may lower $\dim(V\cap\mathcal P_t)$ away from the central fibre. This obstructs the constant-rank continuation method, but does not by itself give a negative answer to \cref{q:balanced-locus}: another class in $\mathcal B_t\cap V$ may still exist.
\end{remark}

\subsection{Related compact deformation theory}
\label{subsec:compact-literature}

We briefly relate this approach to existing results on compact deformations. Let $(X_0,\omega_0)$ be balanced and put
\[
 c=[\omega_0^{n-1}]_{\DR}.
\]
Following Popovici, a nearby fibre is said to be co-polarized by $c$ when this fixed de Rham class is of type $(n-1,n-1)$ for the complex structure of the fibre \cite[Definition~4.1]{Pop19}. For a balanced $\ddbar$-manifold, \cite[Observation~7.2]{Pop19} shows that, along sufficiently small co-polarized deformations, the same class $c$ contains a strictly positive $d$-closed $(n-1,n-1)$-representative. If $\omega_0$ is balanced hyperbolic, then $c\in V$, so \cref{prop:intersection} transfers the bounded-primitive property to Popovici's nearby representative. Combining Popovici's result with \cref{lem:cohomological} therefore gives balanced hyperbolicity on these nearby fibres. The co-polarized result is naturally interpreted as the fixed-section case $c_t\equiv c_0$ of the Aeppli realization of the moving-intersection template.

Sferruzza derived first-order necessary conditions for extending a prescribed balanced
metric \cite{Sfe22}; for the present problem, the class-valued linearization is elementary, while its Aeppli deformation term is computed in \cref{prop:first-variation}. Hu--Xia's canonical Aeppli deformations and jumping formulas \cite{HX25} provide an all-order framework for tracking the moving Aeppli spaces and their possible dimension jumps. The additional feature in \cref{thm:constant-rank} is that the obstruction map is restricted to the fixed $\widetilde d$-bounded subspace $V$, and the continued class is allowed to drift only inside $V\cap\mathcal P_t$. Xia's lower semicontinuity theorem for the balanced cone \cite{Xia25} controls the existence and positivity of nearby balanced classes, but does not by itself force their de Rham classes to lie in $V$. Likewise, Angella--Ugarte's criteria in terms of strongly Gauduchon cones and Bott--Chern/Aeppli numerical data enlarge the known balanced locus, while the intersection with the fixed hyperbolic subspace remains an additional requirement \cite{AU17}.

\section{Lefschetz saturation and topological consequences}
\label{sec:topological-program}

Our second approach avoids tracking a moving pure-type subspace: instead, it forces the fixed hyperbolic subspace to fill an entire cohomology degree. We combine the known graded-ideal property of hyperbolic cohomology \cite[Remark~4.5]{BKS24} with hard Lefschetz surjectivity. Related topological aspects of hyperbolic cohomology are studied in \cite{BKS24}, while higher-degree notions such as K\"ahler $k$-hyperbolicity are developed in \cite{FWW26}. The elementary saturation statement below has useful deformation-theoretic consequences.

\begin{proposition}\label{prop:lefschetz-saturation}
Let $M^{2n}$ be a compact smooth manifold, let $a\in H^2_{\DR}(M,\R)$, and let $r\geq1$. Suppose
\[
 a^r\in\Vhyp^{2r}(M).
\]
For an integer $q$ with $2r\leq q\leq2n$, if
\[
L_a^r:H^{q-2r}_{\DR}(M,\R)\longrightarrow H^q_{\DR}(M,\R),\qquad b\longmapsto a^r\smile b,
\]
is surjective, then
\[
 \Vhyp^q(M)=H^q_{\DR}(M,\R).
\]
If, in addition, $a$ is represented by a symplectic form satisfying hard Lefschetz, the required surjectivity holds whenever
\[
 q\geq n+r.
\]
\end{proposition}

\begin{proof}
By \cref{prop:graded-ideal}, for every $b\in H^{q-2r}_{\DR}(M,\R)$ one has
\[
 a^r\smile b\in\Vhyp^q(M).
\]
Surjectivity therefore gives $H^q_{\DR}(M,\R)\subseteq\Vhyp^q(M)$, while the reverse inclusion is tautological.

For the last assertion, write $m=q-2r$. Surjectivity of $L_a^r:H^m\to H^{m+2r}$ is dual, under Poincar\'e duality, to injectivity of
\[
 L_a^r:H^{2n-m-2r}_{\DR}(M,\R)\longrightarrow H^{2n-m}_{\DR}(M,\R).
\]
If $q=m+2r\geq n+r$, then $p:=2n-m-2r=2n-q\leq n-r$. For $\alpha\in H^p_{\DR}(M,\R)$ with $a^r\smile\alpha=0$, multiplication by $a^{n-r-p}$ gives
\[
 a^{n-p}\smile\alpha=0.
\]
The hard Lefschetz theorem makes $L_a^{n-p}:H^p_{\DR}(M,\R)\to H^{2n-p}_{\DR}(M,\R)$ an isomorphism, so $\alpha=0$. Thus the dual map is injective and $L_a^r:H^m\to H^{m+2r}$ is surjective.
\end{proof}

For $q=2n-2$ and $1\le r\le n-2$, the inequality $q\ge n+r$ holds. Consequently, if $M$ carries a hard-Lefschetz symplectic class $a$ with $a^r\in\Vhyp^{2r}(M)$, then
$\Vhyp^{2n-2}(M)=H^{2n-2}_{\DR}(M,\R)$. Every balanced complex structure on $M$ is then balanced hyperbolic by \cref{prop:intersection}. We apply this observation next.

\begin{corollary}[=\cref{t14}]
\label{cor:KH-balanced-locus}
Let $X_0=(X,J_0)$ be a compact K\"ahler $k$-hyperbolic manifold of complex dimension $n\geq3$, where
\[
 1\leq k\leq n-2.
\]
Then
\[
 \Vhyp^{2n-2}(X)=H^{2n-2}_{\DR}(X,\R)
\]
for the underlying smooth manifold $X$. Consequently every balanced complex structure on $X$ is balanced hyperbolic. In particular, every sufficiently small holomorphic deformation $X_t$ of $X_0$ is balanced hyperbolic.
\end{corollary}

\begin{proof}
Choose a K\"ahler form $\omega$ on $X_0$ such that $\omega^k$ is $\widetilde d$-bounded, and put $a=[\omega]_{\DR}$. Then $a^k\in\Vhyp^{2k}(X)$. The K\"ahler class $a$ satisfies hard Lefschetz, so \cref{prop:lefschetz-saturation} applies with $r=k$. Finally, small deformations of a compact K\"ahler manifold remain K\"ahler \cite{KS60}; hence the nearby fibres are balanced and the preceding assertion applies.
\end{proof}

Moishezon deformation theory can also supply balancedness throughout a family. Recall that every Moishezon manifold belongs to Fujiki class $\mathcal C$ and hence is balanced \cite{AB93}.

\begin{theorem}\label{thm:whole-disc}
Let $\pi:\mathcal X\to\Delta$ be a proper holomorphic submersion over a disc with connected compact fibres of complex dimension $n\ge3$, and let $M$ be their common underlying smooth manifold.
\begin{enumerate}[label=\textup{(\alph*)}]
\item If $\Vhyp^{2n-2}(M)=H^{2n-2}_{\DR}(M,\R)$ and uncountably many fibres are Moishezon, then all fibres are balanced hyperbolic. The saturation hypothesis holds, in particular, if one fibre is K\"ahler $k$-hyperbolic for some $1\le k\le n-2$.
\item If one fibre $X_{t_0}$ is K\"ahler hyperbolic, then every fibre $X_t$ over the entire disc is Moishezon and balanced hyperbolic.
\end{enumerate}
\end{theorem}
\begin{proof}
\emph{(a).} By \cite[Theorem~1.5]{Rao26}, every fibre is Moishezon and hence balanced \cite{AB93}. The balanced form $\omega_t^{n-1}$ on any fibre represents a class of
$\Vhyp^{2n-2}(M)=H^{2n-2}_{\DR}(M,\R)$ and is therefore $\widetilde d$-bounded by \cref{lem:cohomological}. The final assertion uses the saturation consequence of \cref{prop:lefschetz-saturation}.

\emph{(b).} Put $\mathcal L=K_{\mathcal X/\Delta}$. By \cite[Theorem~2.11]{CY18}, $K_{X_{t_0}}$ is ample. A positive metric extends locally to $\mathcal L$, so $K_{X_t}$ is ample and big for $t$ in a neighbourhood of $t_0$. Since $\mathcal L$ exists over the
whole disc, \cite[Corollary~4.3]{RT21} implies that $K_{X_t}$ is big for every $t\in\Delta$. Thus all fibres are Moishezon and balanced. Applying the saturation consequence of \cref{prop:lefschetz-saturation} with $r=1$ to the K\"ahler-hyperbolic fibre gives $\Vhyp^{2n-2}(M)=H^{2n-2}_{\DR}(M,\R)$. The conclusion follows from \cref{prop:intersection}.
\end{proof}

The two parts use different algebraicity results: part~\textup{(a)} invokes Rao's Moishezon propagation theorem, while part~\textup{(b)} uses canonical ampleness and Rao--Tsai's extension of bigness. They assert balanced, rather than ordinary K\"ahler, hyperbolicity of all fibres; neither assumes constancy of Aeppli numbers.

The same saturation argument also applies to higher powers of K\"{a}hler forms.

\begin{corollary}\label{cor:kahler-k-upper}
Let $X_0=(X,J_0)$ be a compact K\"ahler $r$-hyperbolic manifold of complex dimension $n$. Let $\ell$ be an integer satisfying
\[
 r\leq\ell\leq n,\qquad 2\ell\geq n+r.
\]
Then
\[
 \Vhyp^{2\ell}(X)=H^{2\ell}_{\DR}(X,\R).
\]
Consequently every K\"ahler complex structure on $X$ is K\"ahler $\ell$-hyperbolic. In particular, every sufficiently small holomorphic deformation $X_t$ of $X_0$ is K\"ahler $\ell$-hyperbolic.
\end{corollary}

\begin{proof}
Choose a K\"ahler form $\alpha$ on $X_0$ such that $\alpha^r$ is $\widetilde d$-bounded, and put $a=[\alpha]_{\DR}$. Then $a^r\in\Vhyp^{2r}(X)$. Apply \cref{prop:lefschetz-saturation} with $q=2\ell$. Since $2\ell\geq n+r$, hard Lefschetz gives surjectivity of
\[
 L_a^r:H^{2\ell-2r}_{\DR}(X,\R)\longrightarrow H^{2\ell}_{\DR}(X,\R),
\]
so $\Vhyp^{2\ell}(X)=H^{2\ell}_{\DR}(X,\R)$.

Now let $J$ be any K\"ahler complex structure on the same smooth manifold and let $\omega$ be a K\"ahler form for $J$. Then $[\omega^\ell]_{\DR}\in H^{2\ell}_{\DR}(X,\R)=\Vhyp^{2\ell}(X)$. By \cref{lem:cohomological}, $\omega^\ell$ is itself $\widetilde d$-bounded, so $\omega$ is K\"ahler $\ell$-hyperbolic. Finally, small deformations of $X_0$ remain K\"ahler \cite{KS60}.
\end{proof}

\begin{remark}\label{rem:kahler-k-ranges}
For a K\"ahler $r$-hyperbolic central fibre, the preceding argument controls the degrees $2\ell$ satisfying $2\ell\geq n+r$. When $2\ell<n+r$, hard Lefschetz does not force
$L_a^r:H^{2\ell-2r}\to H^{2\ell}$ to be surjective, so the ideal argument alone does not saturate all of $H^{2\ell}_{\DR}(X,\R)$.

In particular, to deduce balanced hyperbolicity one takes $\ell=n-1$, and the condition becomes $r\leq n-2$. If $r=n-1$, K\"ahler $(n-1)$-hyperbolicity supplies the bounded class $a^{n-1}\in\Vhyp^{2n-2}(X)$, but the present ideal argument reaches only the image of the relevant map
\[
 L_a^{n-1}:H^0_{\DR}(X,\R)\longrightarrow H^{2n-2}_{\DR}(X,\R)
\]
This image is one-dimensional, so the map is surjective only when $b_{2n-2}(X)=1$. Thus, when $b_{2n-2}(X)>1$, the method leaves open whether additional hyperbolic classes saturate $H^{2n-2}_{\DR}(X,\R)$; it does not by itself prove non-saturation. Accordingly, the present argument does not extend \cref{t14} to $r=n-1$ without an additional hypothesis. The saturation argument is separate from the unrestricted degree-two deformation theorem of \cref{thm:kahler-openness}.
\end{remark}

\section{Balanced hyperbolicity of principal torus bundles}\label{sec:torus-question}

Although \cref{q:balanced-locus} remains open in general, it has an affirmative answer for the non-K\"ahler manifolds considered here, even without assuming balancedness of nearby fibres. The key is an explicit calculation of the degree-$2n-2$ hyperbolic cohomology. Stability of the holomorphic bundle structure under small deformations follows from Catanese's theorem \cite[Proposition~4.2]{Cat04}.

Let
\[
Y=C_1\times C_2,\qquad g(C_1),g(C_2)\geq2,
\]
where $\kappa_i$ is a positive multiple of the hyperbolic area form on $C_i$, normalized by $\int_{C_i}\kappa_i=1$. We also use $\kappa_i$ for its pullback to $Y$. Set
\[
u_i=[\kappa_i]\in H^2(Y,\R),\qquad U=\R u_1\oplus\R u_2,\qquad K_{\mathrm{mix}}=H^1(C_1,\R)\otimes H^1(C_2,\R).
\]
The K\"unneth decomposition is
\begin{equation}\label{eq:torus-Kunneth}
H^2(Y,\R)=U\oplus K_{\mathrm{mix}},\qquad U^{\perp}=K_{\mathrm{mix}}
\end{equation}
with orthogonality taken for the intersection pairing on $Y$. Throughout this section $f:M\to Y$ is a smooth principal $T^{2r}$-bundle, $r\geq1$, and $f_!$ denotes integration over the oriented real $2r$-dimensional fibre.

\subsection{Hyperbolic cohomology of the total space}\label{subsec:torus-V}

\begin{lemma}\label{lem:torus-bounded-homotopy}
On $\widetilde Y=\mathbb H\times\mathbb H$, every closed smooth $3$-form with bounded covariant derivatives of all orders has a bounded smooth $2$-form primitive with bounded covariant derivatives. The analogous assertion holds in degree $4$, with a primitive of degree $3$.
\end{lemma}
\begin{proof}
In an upper half-plane factor, a bounded top-degree form can be written
\[
\zeta=h(x,y)\frac{dx\wedge dy}{y^2},\qquad \|h\|_\infty<\infty.
\]
Define
\begin{equation}\label{eq:halfplane-bounded-homotopy}
H\zeta=\left(\int_y^\infty\frac{h(x,s)}{s^2}\,ds\right)dx.
\end{equation}
Then $dH\zeta=\zeta$ and $\|H\zeta\|_{\infty}\leq\|h\|_{\infty}$ for the hyperbolic metric. The same expression works with parameters in the other factor and commutes, with the standard graded sign, with differentiation in that factor. Repeated differentiation in hyperbolic orthonormal frames gives the asserted derivative bounds when the input derivatives are bounded.

Decompose a closed $3$-form as $R=R_{2,1}+R_{1,2}$ according to the real degrees in the two factors. Apply $H$ in the first factor to $R_{2,1}$, and subtract the differential of the resulting $(1,1)$-form. The remainder is closed and has degree $(1,2)$. Apply $H$ in the second factor; closedness ensures that no unwanted first-factor derivative is introduced. This yields a bounded primitive. For a $4$-form, which has degree $(2,2)$, applying $H$ in the first factor suffices. The bounds persist at each step.
\end{proof}

\begin{proposition}\label{prop:torus-hyp-cohom}
For a smooth principal $T^{2r}$-bundle $f:M\to C_1\times C_2$ as above, one has
\begin{equation}\label{eq:torus-hyp-subspace}
\Vhyp^{2r+2}(M)=f_!^{-1}(U).
\end{equation}
In particular, with annihilators taken under Poincar\'e duality,
\begin{equation}\label{eq:torus-hyp-annihilator}
\bigl(\Vhyp^{2r+2}(M)\bigr)^\perp=f^*K_{\mathrm{mix}}\subset H^2(M,\R).
\end{equation}
\end{proposition}
\begin{proof}
First suppose $c\in\Vhyp^{2r+2}(M)$. Let $F:T^2\to Y$ be the product of a loop in $C_1$ and a loop in $C_2$, and form the pullback principal torus bundle $P=F^*M\to T^2$. The fundamental group of $P$ is a central extension of $\Z^2$ by $\Z^{2r}$, hence is amenable. The pullback of $c$ to $P$ remains hyperbolic (the map $P\to M$ lifts with bounded derivative, since both bases are compact), so it vanishes by the amenable-fundamental-group vanishing theorem \cite[Theorem~3.2]{BKS24}. Therefore
\[
0=\int_P c=\int_{T^2}F^*(f_!c).
\]
Such product tori detect the entire mixed K\"unneth summand. Hence $f_!c\in U$.

Conversely, let $f_!c\in U$. Choose a smooth closed representative $\Omega$ of $c$ and average it under the torus action. Translations act trivially on de Rham cohomology, so this does not change $c$. Choose invariant real connection forms $\theta_1,\ldots,\theta_{2r}$ with $d\theta_j=f^*F_j$. An invariant $(2r+2)$-form decomposes into horizontal--vertical degrees
\[
(2,2r),\qquad(3,2r-1),\qquad(4,2r-2).
\]
Its highest vertical-degree coefficient is $a=f_!\Omega$, a smooth closed $2$-form on $Y$ with $[a]\in U$. Thus $a$ is cohomologous to $\lambda_1\kappa_1+\lambda_2\kappa_2$ for real $\lambda_i$, and on $\widetilde Y$ it has a bounded primitive $b$ with bounded derivatives: use the bounded primitives of the two hyperbolic area forms and the lift of a smooth exact correction from $Y$.

Let $\widetilde f:\widetilde M\to\widetilde Y$ be the lifted bundle map. Suppressing pullbacks, set
\[
\Gamma_0=b\wedge\theta_1\wedge\cdots\wedge\theta_{2r}.
\]
The closed remainder $R=\widetilde\Omega-d\Gamma_0$ has horizontal degrees only $3$ and $4$. Its horizontal-degree-$3$ coefficients are closed $3$-forms on $\widetilde Y$: this follows from $dR=0$ by looking at horizontal degree $4$ and vertical degree $2r-1$. The coefficients and their derivatives are bounded, since the input forms and curvatures are lifted from compact bases, and $b$ has the stated bounds. By \cref{lem:torus-bounded-homotopy}, they have bounded primitives. Wedging these primitives with the appropriate products of connection forms yields a bounded form $\Gamma_1$ that cancels the horizontal-degree-$3$ remainder.

The still-closed remainder $R-d\Gamma_1$ has only horizontal degree $4$. Apply the same lemma coefficientwise to find a bounded correction $\Gamma_2$. Differentiating its connection factors would have horizontal degree at least $6$ and hence vanishes on the real four-dimensional base. Consequently,
\[
\widetilde\Omega=d(\Gamma_0+\Gamma_1+\Gamma_2),
\]
and this primitive is bounded for a lifted metric. This proves \eqref{eq:torus-hyp-subspace}. The annihilator statement follows from the projection formula $\int_M f^*\alpha\wedge c=\int_Y\alpha\wedge f_!c$ and finite-dimensional duality applied to $f_!:H^{2r+2}(M)\to H^2(Y)$.
\end{proof}

\subsection{A numerical criterion for balanced hyperbolicity}\label{subsec:torus-criterion}

Suppose now that $f:X\to Y$ is a principal holomorphic bundle with complex torus fibre of dimension $r$. Fix a basis of the fibre lattice and let
\[
e_j\in H^2(Y,\Z),\qquad 1\leq j\leq2r,
\]
be the characteristic classes of the underlying real torus bundle. Decompose their images in real cohomology uniquely as
\begin{equation}\label{eq:torus-characteristic-degrees}
e_j=p_ju_1+q_ju_2+\chi_j,\qquad \chi_j\in K_{\mathrm{mix}}.
\end{equation}

\begin{theorem}\label{thm:torus-numerical}
With the notation above,
\begin{equation}\label{eq:torus-numerical}
X\text{ is balanced hyperbolic}\quad\Longleftrightarrow\quad\exists\,\lambda_1,\lambda_2>0:\ \lambda_1q_j+\lambda_2p_j=0\ \text{for all }j.
\end{equation}
If the numerical condition fails, there exists a nonzero smooth, semipositive, closed real $(1,1)$-form $T$ with $[T]_{\DR}\in(\Vhyp^{2r+2}(X))^\perp=f^*K_{\mathrm{mix}}$.
\end{theorem}
\begin{proof}
Assume first that $X$ admits a strictly positive closed hyperbolic $(r+1,r+1)$-form $\Omega$. Integration over the complex torus fibres produces a positive closed $(1,1)$-form $a=f_!\Omega$ on $Y$. By \cref{prop:torus-hyp-cohom}, $[a]=\lambda_1u_1+\lambda_2u_2$; positivity on each coordinate curve implies $\lambda_1,\lambda_2>0$. Each $f^*e_j$ vanishes in de Rham cohomology, since $e_j$ is represented by the curvature of a connection form on $X$. The projection formula gives
\[
[a]\smile e_j=f_!\bigl([\Omega]\smile f^*e_j\bigr)=0\in H^4(Y,\R).
\]
Using $u_i^2=0$, $u_1u_2\ne0$ and $u_i\smile K_{\mathrm{mix}}=0$ gives precisely $\lambda_1q_j+\lambda_2p_j=0$.

Conversely, suppose such positive $\lambda_i$ exist. Choose
\[
a=\lambda_1\kappa_1+\lambda_2\kappa_2,
\]
a K\"ahler form on $Y$ with $[a]\smile e_j=0$ for all $j$. By averaging a smooth Hermitian splitting of the holomorphic tangent sequence over the compact fibre torus, choose translation-invariant complex connection $1$-forms $\vartheta^1,\ldots,\vartheta^r$ of type $(1,0)$, whose restrictions form a basis of vertical $(1,0)$-forms. Their differentials are basic and have types $(2,0)+(1,1)$:
\[
d\vartheta^\alpha=f^*F^\alpha,
\]
where each cohomology class $[F^\alpha]$ is a complex-linear combination of the characteristic classes $e_j$. Hence
\[
\int_Y a\wedge F^\alpha=0.
\]
On the compact K\"ahler surface $(Y,a)$, the complex scalar elliptic operator $u\mapsto a\wedge\partial\bar\partial u$ maps smoothly onto the top-degree forms of integral zero. Thus there are smooth complex functions $u^\alpha$ such that, after replacing $\vartheta^\alpha$ by $\vartheta^\alpha+f^*\partial u^\alpha$, the curvature forms satisfy
\begin{equation}\label{eq:torus-primitive-curvature}
a\wedge F^\alpha=0\qquad (1\leq\alpha\leq r).
\end{equation}

Put
\[
\eta=\frac{i}{2}\sum_{\alpha=1}^r\vartheta^\alpha\wedge\overline{\vartheta^\alpha},\qquad
\omega=f^*a+\eta.
\]
The Hermitian form $\omega$ is positive. Moreover, \eqref{eq:torus-primitive-curvature} and its conjugate give $a\wedge d\eta=0$. Since $a^3=0$ and $\eta^{r+1}=0$, one has
\begin{equation}\label{eq:torus-balanced-form}
\Omega:=\omega^{r+1}=(r+1)f^*a\wedge\eta^r+\binom{r+1}{2}f^*(a^2)\wedge\eta^{r-1},\qquad d\Omega=0.
\end{equation}
This is strictly positive of bidegree $(r+1,r+1)$. Integration along the fibre gives $f_!\Omega=c\,a$ for a constant $c>0$, because the vertical volume $\eta^r$ is translation invariant. Hence $[f_!\Omega]\in U$ and \cref{prop:torus-hyp-cohom} shows that $[\Omega]\in\Vhyp^{2r+2}(X)$. By \cref{lem:cohomological}, $\Omega$ has a bounded primitive upstairs, proving balanced hyperbolicity.

Finally, suppose the numerical condition fails. Let $E\subset H^2(Y,\R)$ be the span of the characteristic classes and let $S\subset\R^2$ be its projection to the coefficient plane $U$. The elementary alternative for a real subspace $S\subset\R^2$ says that either its annihilator contains a vector with both coordinates strictly positive, or $S$ contains a nonzero vector $(\mu_1,\mu_2)$ with $\mu_i\geq0$. Failure of \eqref{eq:torus-numerical} yields the latter. Thus there is an $e\in E$ with
\[
e=\mu_1u_1+\mu_2u_2+\chi,\qquad \chi\in K_{\mathrm{mix}},\quad \mu_i\geq0,\quad (\mu_1,\mu_2)\ne(0,0).
\]
Since $f^*e=0$ in real cohomology, the nonzero semipositive closed form
\[
T=f^*(\mu_1\kappa_1+\mu_2\kappa_2)
\]
satisfies $[T]=-f^*\chi\in f^*K_{\mathrm{mix}}$, as required.
\end{proof}

The sufficiency argument also yields a convenient general observation: if a K\"ahler form $a$ on $Y$ satisfies $[a]\smile e_j=0$ for every $j$, then the same curvature adjustment constructs a balanced metric on $X$. The extra requirement $[a]\in U$ is what gives hyperbolicity via \cref{prop:torus-hyp-cohom}.

\subsection{Small deformations and positive closed obstructions}\label{subsec:torus-deform}

\begin{theorem}\label{thm:torus-deform}
Let $X_0$ be a principal holomorphic complex-torus bundle over $C_1\times C_2$, where both curves have genus at least two. If $X_0$ is balanced hyperbolic, then every sufficiently small holomorphic deformation of the total space $X_0$ is balanced hyperbolic. In particular, \cref{q:balanced-locus} has an affirmative answer within this class, without assuming balancedness of the nearby fibres.
\end{theorem}
\begin{proof}
There are no holomorphic vector fields on $Y=C_1\times C_2$, so $H^0(Y,T_Y)=0$. Catanese's stability theorem \cite[Proposition~4.2]{Cat04} asserts that every sufficiently small deformation of $X_0$ remains a principal holomorphic torus bundle over a small deformation $Y_t$ of $Y$. Small deformations of a product of curves of genus at least two remain products $Y_t=C_{1,t}\times C_{2,t}$: the cross terms in $H^1(Y,T_Y)$ vanish because $H^0(C_i,T_{C_i})=0$, and the product of the two unobstructed curve-deformation families is locally complete (cf.\ \cite{vO05}).

Use differentiable markings on the nearby bases and on the fibre lattice. The characteristic classes $e_j$ are locally constant integral classes; their components $p_j,q_j$ in \eqref{eq:torus-characteristic-degrees} are therefore unchanged. Since $X_0$ is balanced hyperbolic, \cref{thm:torus-numerical} supplies $\lambda_1,\lambda_2>0$ satisfying all the linear equations. The same positive coefficients satisfy the equations on each nearby bundle, so \cref{thm:torus-numerical} applies to every sufficiently small $X_t$.
\end{proof}

These bundles also exclude a particular type of cohomological obstruction. If $X_0$ is balanced hyperbolic, a nonzero positive $d$-closed $(1,1)$-form on an arbitrarily close fibre whose de Rham class annihilates the fixed hyperbolic subspace would give
\[
0=\int_{X_t}T\wedge\Omega_0=\int_{X_t}T\wedge(\Omega_0)^{r+1,r+1}_{J_t}>0,
\]
a contradiction (the last projection remains strictly positive for small $t$). For general balanced deformations, the obstruction may be only Aeppli-cohomological rather than represented by such a closed positive form; this is why the argument does not settle \cref{q:balanced-locus} in general.

\subsection{Non-K\"ahler examples and the sharpness of the condition}\label{subsec:torus-examples}

\begin{example}[Contrasting elliptic-bundle examples]\label{ex:torus-comparison}
\noindent\textup{(a) A balanced-hyperbolic non-K\"ahler bundle.}
Choose line bundles $A_i\to C_i$ of degree one and let
\[
L=\operatorname{pr}_1^*A_1\otimes\operatorname{pr}_2^*A_2^{-1}.
\]
For a fixed real $0<q<1$, the quotient $X=L^\times/\langle q\rangle$ is a principal elliptic bundle over $C_1\times C_2$, with fibre $\C^*/q^\Z$. In a suitable lattice basis its characteristic classes are $e_1=u_1-u_2$ and $e_2=0$. Taking $\lambda_1=\lambda_2>0$ in \eqref{eq:torus-numerical} proves that $X$ is balanced hyperbolic; by \cref{thm:torus-deform}, so is every small deformation of $X$. It is not K\"ahler: the low-degree cohomology sequence of the bundle gives
\[
b_1(X)=2g(C_1)+2g(C_2)+1,
\]
which is odd. Notice also that $h_{\BC}^{2,0}(X)\geq g(C_1)g(C_2)>0$, by pulling back closed holomorphic two-forms from the base. Therefore this example is not covered by the vanishing hypothesis of \cref{thm:aeppli-vanishing}.
\par\medskip\noindent\textup{(b) A balanced but non-hyperbolic bundle.}
Let $Y=C\times C$ with $g(C)=g\geq2$, and let $\Delta_C\subset Y$ be its diagonal. Set
\[
\chi=[\Delta_C]-u_1-u_2\in K_{\mathrm{mix}}\cap H^2(Y,\Z).
\]
The intersection numbers of the diagonal give $\chi^2=-2g$. Choose a smooth real closed $(1,1)$-form $\zeta$ representing $\chi$. For sufficiently large integers $N>0$, the form
\[
a_N=\kappa_1+\kappa_2+\frac{1}{gN}\zeta
\]
is K\"ahler. The integral $(1,1)$-class
\[
e_N=u_1+u_2+N\chi
\]
is the Chern class of a holomorphic line bundle $L_N\to Y$, and
\[
[a_N]\smile e_N=2+\frac1g\chi^2=0.
\]
Form the principal elliptic bundle $X_N=L_N^\times/\langle q\rangle$, $0<q<1$. The balanced construction in the proof of \cref{thm:torus-numerical}, applied to $a_N$, produces a balanced metric on $X_N$. But the characteristic-degree vectors are $(p_1,q_1)=(1,1)$ and $(p_2,q_2)=(0,0)$, so \eqref{eq:torus-numerical} would require $\lambda_1+\lambda_2=0$ with both numbers positive. Hence $X_N$ is not balanced hyperbolic. Explicitly, the semipositive closed form
\[
T=f^*(\kappa_1+\kappa_2),\qquad [T]=-Nf^*\chi\in f^*K_{\mathrm{mix}},
\]
is the obstruction of \cref{thm:torus-numerical}.
\end{example}

The examples show that the numerical criterion is nontrivial: even among principal elliptic bundles, balancedness does not imply balanced hyperbolicity. They are not counterexamples to \cref{q:balanced-locus}: their characteristic-degree data cannot change in a small deformation that retains the principal torus-bundle structure.

\appendix
\section{Bounded correction problems on the universal cover}
\label{sec:universal-program}

\subsection{Motivation}

Fu--Yau proved persistence of balanced structures under their weak $(n-1,n)$-$\ddbar$ condition \cite[Theorem~6]{FY11}. Under the corresponding hypotheses, Khelifati proves that the nearby balanced metrics obtained from a balanced hyperbolic central metric are Gauduchon hyperbolic \cite[Theorem~2.15(1)]{Khe26}; equivalently, their lifted $(n-1,n-1)$ powers admit bounded $(\partial+\dbar)$-potentials of the appropriate bidegrees. This motivates the question of whether such a potential can be corrected to a bounded $d$-primitive. Khelifati formulates
\[
\text{balanced and Gauduchon hyperbolic}\quad\Longrightarrow\quad\text{balanced hyperbolic}
\]
as Conjecture~2.20 and explains that the required complete manifold equations with
$L^{\infty}$-control are not presently available \cite[Conjecture~2.20]{Khe26}.

\Cref{prop:top-row} reduces the correction problem to a top-row $\partial$-equation. The choice of potential matters: one bounded potential may admit a correction while another does not. Moreover, boundedness of $\gamma$ gives no $L^{\infty}$ estimate for $\dbar\gamma$. We formulate a choice-independent quotient obstruction and give sufficient conditions for its vanishing.

The divisor-mass estimate in \cite[Proposition~3.4]{Rao26} has a different purpose: it bounds a scalar Gauduchon pairing, without estimating the full correcting form. It therefore does not supply either the small smooth corrections used for positivity in \cref{sec:compact-program,sec:monodromy} or the bounded primitives on the universal cover sought here.

\subsection{The bounded top-row equation}
\label{subsec:top-row}

Let $A^{p,q}_{\mathrm b}(\widetilde X)$ denote the vector space of smooth $(p,q)$-forms
whose pointwise norm is bounded for a lifted Hermitian metric. This notation imposes no boundedness condition on their derivatives. For $n\geq3$, define the bounded-primitive
quotient
\begin{equation}\label{eq:bounded-quotient}
\mathcal Q^{n-2,n}_{\partial,\mathrm b}(\widetilde X):=\frac{\ker\bigl(\partial:A^{n-2,n}(\widetilde X)\to A^{n-1,n}(\widetilde X)\bigr)}{\partial A^{n-3,n}_{\mathrm b}(\widetilde X)}.
\end{equation}
The numerator is not required to consist of bounded forms: the quotient records whether
a given $\partial$-closed form admits a bounded $\partial$-primitive.

\begin{proposition}\label{prop:top-row}
Let $\Omega$ be a real, $d$-closed $(n-1,n-1)$-form on a compact complex $n$-fold $X$.
Suppose its lift to the universal cover admits a bounded $(\partial+\dbar)$-potential. After symmetrizing, this means that one can write
\begin{equation}\label{eq:symmetric-GH}
\widetilde\Omega=\partial\gamma+\dbar\bar\gamma,\qquad \gamma\in A^{n-2,n-1}_{\mathrm b}(\widetilde X).
\end{equation}
If $n\geq3$, then $\widetilde\Omega$ has a bounded real $d$-primitive if and only if one
can choose $\gamma$ in \eqref{eq:symmetric-GH} and find $u\in A^{n-3,n}_{\mathrm b}(\widetilde X)$ such that
\begin{equation}\label{eq:top-row}
\partial u=-\dbar\gamma.
\end{equation}
Equivalently, zero belongs to
\[
\mathfrak G_{\infty}(\widetilde\Omega):=\left\{[\dbar\gamma]\in\mathcal Q^{n-2,n}_{\partial,\mathrm b}(\widetilde X):
\begin{array}{l}
\gamma\in A^{n-2,n-1}_{\mathrm b}(\widetilde X),\\[-2pt]
\widetilde\Omega=\partial\gamma+\dbar\bar\gamma
\end{array}\right\}.
\]
For $n=2$, a bounded real $d$-primitive exists if and only if there is a bounded $\gamma\in A^{0,1}(\widetilde X)$ satisfying \eqref{eq:symmetric-GH} and $\dbar\gamma=0$.

For $n\geq3$, fix one bounded symmetric potential $\gamma_0$ satisfying \eqref{eq:symmetric-GH}. Regard $\mathcal Q^{n-2,n}_{\partial,\mathrm b}(\widetilde X)$ as a real vector space, and define the real vector space of bounded null-potentials
\[
\mathcal N_{\mathrm b}:=\left\{\delta\in A^{n-2,n-1}_{\mathrm b}(\widetilde X):\partial\delta+\dbar\bar\delta=0\right\}
\]
and the real-linear map
\[
\Phi:\mathcal N_{\mathrm b}\longrightarrow\mathcal Q^{n-2,n}_{\partial,\mathrm b}(\widetilde X),\qquad\Phi(\delta):=[\dbar\delta].
\]
Then
\begin{equation}\label{eq:affine-obstruction-set}
\mathfrak G_{\infty}(\widetilde\Omega)=[\dbar\gamma_0]+\operatorname{Im}\Phi.
\end{equation}
Consequently,
\begin{equation}\label{eq:canonical-obstruction}
\operatorname{Ob}_{\infty}(\widetilde\Omega):=\bigl[[\dbar\gamma_0]\bigr]\in\frac{\mathcal Q^{n-2,n}_{\partial,\mathrm b}(\widetilde X)}{\operatorname{Im}\Phi}
\end{equation}
is independent of the choice of $\gamma_0$, and
\[
\widetilde\Omega\text{ has a bounded real }d\text{-primitive}\quad\Longleftrightarrow\quad\operatorname{Ob}_{\infty}(\widetilde\Omega)=0.
\]

\end{proposition}
\begin{proof}
We first make the symmetrization explicit. If a bounded $(\partial+\dbar)$-potential is initially written
\[
\widetilde\Omega=\partial\alpha+\dbar\beta,\qquad\alpha\in A^{n-2,n-1}_{\mathrm b}(\widetilde X),\quad\beta\in A^{n-1,n-2}_{\mathrm b}(\widetilde X),
\]
then reality of $\widetilde\Omega$ gives the conjugate identity as well. Hence, with
$\gamma=\tfrac12(\alpha+\bar\beta)$, averaging the two identities yields $\widetilde\Omega=\partial\gamma+\dbar\bar\gamma$, and $\gamma$ is bounded.

Because $d\widetilde\Omega=0$, the $(n-1,n)$-component of $d\widetilde\Omega$ gives $\partial\dbar\gamma=0$, so $\dbar\gamma$ represents a class in \eqref{eq:bounded-quotient}. If \eqref{eq:top-row} has a bounded solution, set
\[
\Gamma=u+\gamma+\bar\gamma+\bar u.
\]
The $(n-2,n)$-component of $d\Gamma$ is $\partial u+\dbar\gamma=0$, and its conjugate $(n,n-2)$-component also vanishes. The $(n-1,n-1)$-component is $\partial\gamma+\dbar\bar\gamma=\widetilde\Omega$. Hence $d\Gamma=\widetilde\Omega$, and $\Gamma$ is bounded.

Conversely, suppose $d\Gamma=\widetilde\Omega$ for a bounded real $(2n-3)$-form $\Gamma$. For $n\geq3$, its only possible type components are
\[
\Gamma=u+\gamma+\bar\gamma+\bar u,\qquad u\in A^{n-3,n},\quad \gamma\in A^{n-2,n-1},
\]
and all components are bounded. Comparing the $(n-1,n-1)$- and $(n-2,n)$-components of
$d\Gamma=\widetilde\Omega$ gives \eqref{eq:symmetric-GH} and \eqref{eq:top-row}. When $n=2$, the $u$-components do not exist, so the off-type equation reduces to $\dbar\gamma=0$.

For the affine and choice-independent statements, let $\delta\in\mathcal N_{\mathrm b}$. Applying $\dbar$ to $\partial\delta+\dbar\bar\delta=0$ gives $\partial\dbar\delta=0$. Hence $[\dbar\delta]$ is defined in \eqref{eq:bounded-quotient}, so $\Phi$ is well-defined.

Every other bounded symmetric potential has the form $\gamma_0+\delta$ with $\delta\in\mathcal N_{\mathrm b}$, and every such sum is again a bounded symmetric
potential. Therefore
\[
[\dbar(\gamma_0+\delta)]=[\dbar\gamma_0]+\Phi(\delta),
\]
which proves \eqref{eq:affine-obstruction-set}. Replacing $\gamma_0$ by another potential changes $[\dbar\gamma_0]$ by an element of $\operatorname{Im}\Phi$; hence \eqref{eq:canonical-obstruction} is choice-independent. Finally, the first part shows that a bounded real $d$-primitive exists exactly when $0\in\mathfrak G_{\infty}(\widetilde\Omega)$, which is equivalent to the vanishing of the quotient class.
\end{proof}

The phrase "one can choose $\gamma$" in \cref{prop:top-row} is important. The proposition does not assert that \eqref{eq:top-row} is solvable for every bounded potential representing $\widetilde\Omega$. Moreover, the definition of $A^{p,q}_{\mathrm b}$ controls only the pointwise norm of $\gamma$, not that of $\dbar\gamma$. Thus an $L^{\infty}$ right inverse for $\partial$ cannot be applied without an additional estimate on the actual datum $-\dbar\gamma$. The next proposition shows that a prescribed bounded potential can be obstructed.

\begin{proposition}\label{prop:prescribed-potential-failure}
There exist a compact complex threefold $X$, a balanced hyperbolic metric $\omega$ on
$X$, and two bounded symmetric potentials $\gamma_0,\gamma_1\in
A^{1,2}_{\mathrm b}(\widetilde X)$ for the same form $\widetilde\omega^{\,2}$ such that
\eqref{eq:top-row} has a bounded solution for $\gamma_0$ but has no bounded solution
for $\gamma_1$.
\end{proposition}

\begin{proof}
Let $E$ be an elliptic curve with a flat K\"ahler form $\alpha$, and let $C_1,C_2$ be compact curves of genus at least two. Put
\[
Y=C_1\times C_2\qquad\textrm{and}\qquad X=E\times Y.
\]
Choose the product K\"ahler hyperbolic form $\kappa$ on $Y$ from \cref{lem:curves-KH}, and let $\theta$ be a bounded one-form on $\widetilde Y$ with $d\theta=\widetilde\kappa$. Put
\[
\omega=\alpha+\kappa.
\]
Since $\alpha^2=0$ on the curve $E$, one has
\[
\widetilde\omega^{\,2}=2\widetilde\alpha\wedge\widetilde\kappa+\widetilde\kappa^{\,2}=d\bigl(2\widetilde\alpha\wedge\theta+\theta\wedge\widetilde\kappa\bigr).
\]
The displayed primitive is bounded, so $\omega$ is balanced hyperbolic. Denote this
primitive by $\Gamma_0$ and decompose it by type as
\[
\Gamma_0=u_0+\gamma_0+\bar\gamma_0+\bar u_0,\qquad u_0\in A^{0,3}_{\mathrm b}(\widetilde X),\quad\gamma_0\in A^{1,2}_{\mathrm b}(\widetilde X).
\]
The type components of $d\Gamma_0=\widetilde\omega^{\,2}$ give
\[
\widetilde\omega^{\,2}=\partial\gamma_0+\dbar\bar\gamma_0\qquad\textrm{and}\qquad\partial u_0=-\dbar\gamma_0.
\]

Identify the universal cover with
\[
\widetilde X=\C_z\times\D_w\times\D_\zeta
\]
and use the product of the flat metric on $\C$ and the Poincar\'e metrics on the two
discs. Consider
\[
\delta:=\bar w\,dz\wedge d\bar z\wedge d\bar\zeta\in A^{1,2}_{\mathrm b}(\widetilde X).
\]
This form is bounded (the coordinate $w$ is bounded and the Poincar\'e norm of
$d\bar\zeta$ is bounded), and a direct differentiation gives
\[
\partial\delta=0\qquad\textrm{and}\qquad\dbar\bar\delta=0.
\]
Thus $\gamma_1:=\gamma_0+\delta$ is another bounded symmetric potential for $\widetilde\omega^{\,2}$.

Suppose that a bounded $u_1\in A^{0,3}_{\mathrm b}(\widetilde X)$ satisfied $\partial u_1=-\dbar\gamma_1$. Then $v:=u_1-u_0$ would be bounded and would satisfy
\begin{equation}\label{eq:bad-null-potential}
\partial v=-\dbar\delta.
\end{equation}
Write
\[
v=h(z,w,\zeta)\,d\bar z\wedge d\bar w\wedge d\bar\zeta.
\]
Since
\[
\dbar\delta=dz\wedge d\bar z\wedge d\bar w\wedge d\bar\zeta,
\]
the $dz$-coefficient in \eqref{eq:bad-null-potential} gives
\[
\frac{\partial h}{\partial z}=-1.
\]
On the slice $w=\zeta=0$, boundedness of $v$ implies that $h_0(z):=h(z,0,0)$ is bounded on $\C$. On the other hand, $\partial h_0/\partial z=-1$. Integrating over the Euclidean disc $D_R\subset\C$ and applying Stokes' theorem gives
\[
\pi R^2=\left|\int_{D_R}\frac{\partial h_0}{\partial z}\,dA\right|\leq C R\,\|h_0\|_{L^{\infty}(\C)},
\]
with an absolute constant $C$, which is impossible as $R\to\infty$. Therefore no bounded $u_1$ exists.

In particular, taking $\widetilde\Omega=0$ shows directly that both $0$ and the non-zero class $[\dbar\delta]$ can occur in $\mathfrak G_{\infty}(0)$; the obstruction set need not be a singleton.
\end{proof}

\begin{remark}
The metric $\omega$ in \cref{prop:prescribed-potential-failure} is already balanced
hyperbolic, so its choice-independent class $\operatorname{Ob}_{\infty}(\widetilde\omega^{\,2})$ vanishes. The proposition shows only that a prescribed bounded potential can be uncorrectable. It neither produces a non-zero choice-independent obstruction nor gives a counterexample to the manifold-level conjecture recalled above.
\end{remark}

\begin{remark}\label{rem:formal-green}
Assume $n\geq3$. Fix a bounded potential $\gamma$ and put $F:=-\dbar\gamma$. Then $\partial F=0$. Suppose, in addition to the hypotheses above, that $F$ lies in a function space on which a $\partial$-Hodge decomposition
\[
F=\mathcal H_{\partial}F+\partial\partial^*G_{\partial}F+\partial^*\partial G_{\partial}F
\]
is valid, that $\mathcal H_{\partial}F=0$, that $G_{\partial}$ commutes with $\partial$ on $F$, and that $\partial^*G_{\partial}F$ is bounded. Since $\partial F=0$, these assumptions give $\partial G_{\partial}F=0$ and hence
\[
F=\partial\partial^*G_{\partial}F.
\]
Thus
\begin{equation}\label{eq:formal-green-solution}
u=\partial^*G_{\partial}F=-\partial^*G_{\partial}(\dbar\gamma)
\end{equation}
solves \eqref{eq:top-row}. The unresolved issue is precisely the availability of such a
global decomposition and, above all, the boundedness estimate in \eqref{eq:formal-green-solution}. Neither follows from boundedness of $\gamma$ alone. For example, one sufficient analytic package is the existence of some $p>2n$ such that $F\in L^p$, $G_{\partial}:L^p\to W^{2,p}$ is bounded on the relevant complement of harmonic forms, and the global embedding $W^{1,p}\hookrightarrow L^{\infty}$ holds; then \eqref{eq:formal-green-solution} lies in $W^{1,p}\subset L^{\infty}$.
\end{remark}

\begin{criterion}\label{crit:compact-residual}
Assume $n\geq3$. Let $p:\widetilde X\to X$ be the universal covering and let $\gamma$ be a bounded symmetric potential. Suppose there exist
\[
U\in A^{n-3,n}_{\mathrm b}(\widetilde X),\qquad R\in A^{n-2,n}(X),\qquad\partial R=0,
\]
such that
\begin{equation}\label{eq:compact-residual}
-\dbar\gamma=\partial U-p^*R.
\end{equation}
If $[R]_{\partial}=0$ in $H^{n-2,n}_{\partial}(X)$, then \eqref{eq:top-row} has a bounded solution. In particular, the conclusion holds whenever $H^{n-2,n}_{\partial}(X)=0$, equivalently whenever $H^{0,2}_{\dbar}(X)=0$.
\end{criterion}

\begin{proof}
Choose $v\in A^{n-3,n}(X)$ with $\partial v=R$. Since $X$ is compact, the lift $p^*v$ is bounded. Equation \eqref{eq:compact-residual} then gives
\[
u:=U-p^*v\in A^{n-3,n}_{\mathrm b}(\widetilde X),\qquad\partial u=-\dbar\gamma.
\]
The final equivalence follows from complex conjugation and Serre duality: $H^{n-2,n}_{\partial}(X)$ is conjugate to $H^{n,n-2}_{\dbar}(X)$, whose dual is $H^{0,2}_{\dbar}(X)$.
\end{proof}

For a balanced Gauduchon-hyperbolic metric $\omega$, the combined criterion above says that $\omega$ is balanced hyperbolic exactly when $\operatorname{Ob}_{\infty}(\widetilde\omega^{\,n-1})=0$, equivalently when at least one bounded symmetric potential admits a bounded solution of \eqref{eq:top-row}. A bounded right inverse on one such datum, the hypotheses of \cref{rem:formal-green}, or the compact residual condition of \cref{crit:compact-residual} suffices. The potential cannot be chosen arbitrarily by
\cref{prop:prescribed-potential-failure}.

The example provides neither a general $L^\infty$ right inverse for $\partial$ nor an estimate for $\dbar\gamma$ from boundedness of $\gamma$ alone.

This fixed-metric statement should be distinguished from Khelifati's \cite[Conjecture~2.20]{Khe26}, which is existential at the level of the manifold: a
compact balanced manifold admitting a balanced Gauduchon-hyperbolic metric is
conjectured to admit some balanced hyperbolic metric, not necessarily the same one.
Vanishing of the fixed-metric obstruction for every balanced Gauduchon-hyperbolic
metric would imply that conjecture, while requiring the prescribed metric itself to
be balanced hyperbolic is the stronger conclusion. The general vanishing problem
remains open.

\section*{Acknowledgments}

The counterexample was constructed with the assistance of the gpt-5.6-sol model and subsequently carefully verified by the authors for mathematical correctness. The second author was supported by the NSFC, Grant No.~12271275.

\bigskip

\noindent
\textsc{Jixiang Fu}\\
Shanghai Center for Mathematical Sciences\\
Fudan University\\
Shanghai 200433, People's Republic of China\\
Email:
\href{mailto:majxfu@fudan.edu.cn}
{\texttt{majxfu@fudan.edu.cn}}
     
\medskip

\noindent
\textsc{Jingcao Wu}\\
School of Mathematics\\
Shanghai University of Finance and Economics\\
Shanghai 200433, People's Republic of China\\
Email:
\href{mailto:wujincao@shufe.edu.cn}
     {\texttt{wujincao@shufe.edu.cn}}


\begin{thebibliography}{99}

\bibitem{AB90}
L.~Alessandrini and G.~Bassanelli,
\emph{Small deformations of a class of compact non-K\"ahler manifolds},
Proc. Amer. Math. Soc. \textbf{109} (1990), no.~4, 1059--1062.
\href{https://doi.org/10.1090/S0002-9939-1990-1012922-5}{doi:10.1090/S0002-9939-1990-1012922-5}.

\bibitem{AB93}
L.~Alessandrini and G.~Bassanelli,
\emph{Metric properties of manifolds bimeromorphic to compact K\"ahler spaces},
J. Differential Geom. \textbf{37} (1993), no.~1, 95--121.

\bibitem{Ati58}
M.~F.~Atiyah,
\emph{On analytic surfaces with double points},
Proc. Roy. Soc. London Ser. A \textbf{247} (1958), 237--244.
\href{https://doi.org/10.1098/rspa.1958.0181}{doi:10.1098/rspa.1958.0181}.

\bibitem{AU17}
D.~Angella and L.~Ugarte,
\emph{On small deformations of balanced manifolds},
Differential Geom. Appl. \textbf{54} (2017), Part~B, 464--474.

\bibitem{BKS24}
M.~Brunnbauer, D.~Kotschick, and L.~Sch\"onlinner,
\emph{On atoroidal and hyperbolic cohomology classes},
Topology Appl. \textbf{344} (2024), Article 108830.
\href{https://doi.org/10.1016/j.topol.2024.108830}{\nolinkurl{doi:10.1016/j.topol.2024.108830}}.

\bibitem{Cat04}
F.~Catanese,
\emph{Deformation in the large of some complex manifolds. I},
Ann. Mat. Pura Appl. (4) \textbf{183} (2004), 261--289.
\href{https://doi.org/10.1007/s10231-003-0096-y}{doi:10.1007/s10231-003-0096-y}.

\bibitem{CY18}
B.-L.~Chen and X.~Yang,
\emph{Compact K\"ahler manifolds homotopic to negatively curved Riemannian manifolds},
Math. Ann. \textbf{370} (2018), no.~3--4, 1477--1489.
\href{https://doi.org/10.1007/s00208-017-1521-7}{doi:10.1007/s00208-017-1521-7}.

\bibitem{Del71}
P.~Deligne,
\emph{Th\'eorie de Hodge II},
Publ. Math. Inst. Hautes \`Etudes Sci. \textbf{40} (1971), 5--57.
\href{https://doi.org/10.1007/BF02684692}{doi:10.1007/BF02684692}.

\bibitem{FWW26}
J.~Fu, H.~Wang, and J.~Wu,
\emph{On the birational invariance of balanced hyperbolic manifolds},
Science China Math. (2026), published online.
\href{https://doi.org/10.1007/s11425-025-2509-1}{doi:10.1007/s11425-025-2509-1}.

\bibitem{FY11}
J.~Fu and S.-T.~Yau,
\emph{A note on small deformations of balanced manifolds},
C. R. Math. Acad. Sci. Paris \textbf{349} (2011), no.~13--14, 793--796.
\href{https://doi.org/10.1016/j.crma.2011.06.023}{doi:10.1016/j.crma.2011.06.023}.

\bibitem{GCartVIII}
A.~Grothendieck,
\emph{Techniques de construction en g\'eom\'etrie analytique. VIII. Rapport sur les th\'eor\`emes de finitude de Grauert et Remmert},
S\'eminaire Henri Cartan \textbf{13} (1960--1961), exp. 15, 10 pp.
\url{https://www.numdam.org/item/SHC_1960-1961__13_2_A2_0/}.

\bibitem{GCartIX}
A.~Grothendieck,
\emph{Techniques de construction en g\'eom\'etrie analytique. IX. Quelques probl\`emes de modules},
S\'eminaire Henri Cartan \textbf{13} (1960--1961), exp. 16, 20 pp.
\url{https://www.numdam.org/item/SHC_1960-1961__13_2_A3_0/}.

\bibitem{Gro91}
M.~Gromov,
\emph{K\"ahler hyperbolicity and $L_2$-Hodge theory},
J. Differential Geom. \textbf{33} (1991), no.~1, 263--292.

\bibitem{HX25}
Y.~Hu and W.~Xia,
\emph{Deformed Aeppli cohomology: canonical deformations and jumping formulas},
Canadian J. Math., First View (2025), 1--34.
\href{https://doi.org/10.4153/S0008414X2510117X}{\nolinkurl{doi:10.4153/S0008414X2510117X}}.

\bibitem{Khe25KH}
A.~Khelifati,
\emph{Holomorphic Deformations of Compact K\"ahler Hyperbolic Manifolds},
arXiv:2508.07096v2, 27 August 2025.

\bibitem{Khe26}
A.~Khelifati,
\emph{Deformations of non-K\"ahler hyperbolicity notions and modifications of
degenerate balanced manifolds},
Int. J. Math. (2026), Article 2650068, 33 pp.
\href{https://doi.org/10.1142/S0129167X26500680}{\nolinkurl{doi:10.1142/S0129167X26500680}}.

\bibitem{KS60}
K.~Kodaira and D.~C.~Spencer,
\emph{On deformations of complex analytic structures. III. Stability theorems for
complex structures},
Ann. of Math. (2) \textbf{71} (1960), 43--76.

\bibitem{LS26}
K.~Liu and Y.~Shen,
\emph{Sections of Hodge bundles II: deformation of $(p,p)$-classes and applications
to K\"ahler geometry},
arXiv:2602.13951, 2026.

\bibitem{Mic82}
M.~L.~Michelsohn,
\emph{On the existence of special metrics in complex geometry},
Acta Math. \textbf{149} (1982), 261--295.

\bibitem{MP23}
S.~Marouani and D.~Popovici,
\emph{Balanced hyperbolic and divisorially hyperbolic compact complex manifolds},
Math. Res. Lett. \textbf{30} (2023), no.~6, 1813--1855.
\href{https://doi.org/10.4310/MRL.2023.v30.n6.a7}{\nolinkurl{doi:10.4310/MRL.2023.v30.n6.a7}}.

\bibitem{Pop15}
D.~Popovici,
\emph{Aeppli cohomology classes associated with Gauduchon metrics on compact complex
manifolds},
Bull. Soc. Math. France \textbf{143} (2015), no.~4, 763--800.

\bibitem{Pop19}
D.~Popovici,
\emph{Holomorphic deformations of balanced Calabi--Yau $\partial\bar\partial$-manifolds},
Ann. Inst. Fourier (Grenoble) \textbf{69} (2019), no.~2, 673--728.
\href{https://doi.org/10.5802/aif.3254}{doi:10.5802/aif.3254}.

\bibitem{Rao26}
S.~Rao,
\emph{Upper semicontinuity of algebraic dimension},
preprint, version of 13 September 2026.

\bibitem{RT21}
S.~Rao and I.-H.~Tsai,
\emph{Deformation limit and bimeromorphic embedding of Moishezon manifolds},
Commun. Contemp. Math. \textbf{23} (2021), no.~8, Paper No.~2050087;
arXiv:1901.10627v3.

\bibitem{Sch07}
M.~Schweitzer,
\emph{Autour de la cohomologie de Bott--Chern},
arXiv:0709.3528, 2007.

\bibitem{Sfe22}
T.~Sferruzza,
\emph{Deformations of balanced metrics},
Bull. Sci. Math. \textbf{178} (2022), Article 103143.
\href{https://doi.org/10.1016/j.bulsci.2022.103143}{doi:10.1016/j.bulsci.2022.103143}.

\bibitem{vO05}
M.~van Opstall,
\emph{Moduli of products of curves},
Arch. Math. (Basel) \textbf{84} (2005), 148--154.
\href{https://doi.org/10.1007/s00013-004-1045-8}{doi:10.1007/s00013-004-1045-8}.

\bibitem{Wu06}
C.-C.~Wu,
\emph{On the geometry of superstrings with torsion},
Ph.D. thesis, Harvard University, 2006.

\bibitem{Xia25}
T.~Xia,
\emph{The deformation of the balanced cone and its degeneration},
Differential Geom. Appl. \textbf{98} (2025), Article 102225.
\href{https://doi.org/10.1016/j.difgeo.2024.102225}{doi:10.1016/j.difgeo.2024.102225}.


\end{thebibliography}
\end{document}